\documentclass[eqn,opre,nonblindrev]{informs4}

\RequirePackage{tgtermes}
\RequirePackage{newtxtext}
\RequirePackage{newtxmath}
\RequirePackage{bm}
\RequirePackage{endnotes}

\OneAndAHalfSpacedXI
\usepackage{tikz}
\usepackage{caption}

\usepackage[labelfont=sf]{subcaption}
\usepackage{hyperref}
\hypersetup{
    colorlinks=true,
    linkcolor=red,
    citecolor=blue,
    urlcolor=blue
}

\usepackage[ruled,vlined,linesnumbered]{algorithm2e}
\usepackage{multirow}
\usepackage{bbm}
\usepackage{cleveref}

\newcommand{\mbR}{\mathbb{R}}
\newcommand{\mbE}{\mathbb{E}}
\newcommand{\mcF}{\mathcal{F}}
\newcommand{\mcS}{\mathcal{S}}

\newcommand{\norm}[1]{\| {#1} \|} 
\newcommand{\bignorm}[1]{\left\| {#1} \right\|} 
\newcommand{\iproduct}[2]{\left\langle {#1}, \; {#2} \right\rangle}
\newcommand{\co}{\operatorname{co}}
\newcommand{\prox}{\operatorname{prox}}

\newcommand{\Var}{\mathbb{V}}
\renewcommand{\Pr}[1]{\mathbb{P}\left\{#1\right\}}
\newcommand{\hnabla}{\hat{\nabla}}

\newcommand{\hG}{\hat{G}}
\newcommand{\hkappa}{\hat{\kappa}}
\newcommand{\Normal}{\mathrm{Normal}}
\newcommand{\opt}{^\star}
\usepackage{tcolorbox}
\usepackage{multibib}
\newcites{ec}{References}

\usepackage{natbib}
 \bibpunct[, ]{(}{)}{,}{a}{}{,}
 \def\bibfont{\small}

\EquationsNumberedThrough 
\TheoremsNumberedThrough     
\ECRepeatTheorems

\begin{document}

\RUNAUTHOR{Zhu, Huang, Chen}

\RUNTITLE{Boosting Accelerated Proximal Gradient Method with Adaptive Sampling for Stochastic Composite Optimization}

\TITLE{Boosting Accelerated Proximal Gradient Method with Adaptive Sampling for Stochastic Composite Optimization \footnote{\textbf{This paper was conducted during Caihua Chen's visit to Stanford University, and the authors thank Prof. Peter W. Glynn (Stanford University) for his invaluable discussions and constructive suggestions that significantly improved the quality of this paper. This work was partially supported by the National Natural Science Foundation of China [Projects 72394363 and 12301601].}}}

\ARTICLEAUTHORS{

\AUTHOR{Dongxuan Zhu}
\AFF{Department of System Engineering and Engineering Management, Chinese University of Hong Kong, Hong Kong, China,
\href{mailto:dxzhu@se.cuhk.edu.hk}{\textcolor{blue}{dxzhu@se.cuhk.edu.hk}}}

\AUTHOR{Weihuan Huang}
\AFF{School of Management \& Engineering, Nanjing University, Nanjing 210093, China, \href{mailto:hwh@nju.edu.cn}{\textcolor{blue}{hwh@nju.edu.cn}}}

\AUTHOR{Caihua Chen}
\AFF{School of Management \& Engineering, Nanjing University, Nanjing 210093, China, \href{mailto:chchen@nju.edu.cn}{\textcolor{blue}{chchen@nju.edu.cn}}}

} 

\ABSTRACT{
We develop an adaptive Nesterov accelerated proximal gradient (adaNAPG) algorithm for stochastic composite optimization problems, boosting the Nesterov accelerated proximal gradient (NAPG) algorithm through the integration of an adaptive sampling strategy for gradient estimation. We provide a complexity analysis demonstrating that the new algorithm, adaNAPG, achieves both the optimal iteration complexity and the optimal sample complexity as outlined in the existing literature. Additionally, we establish a central limit theorem for the iteration sequence of the new algorithm adaNAPG, elucidating its convergence rate and efficiency.
}

\maketitle

\section{Introduction}\label{sec:Intro}

Stochastic composite optimization problem ~\citep{lan2012optimal,li2016majorized,tran2022new} is of considerable importance in large-scale decision-making contexts. It has widespread applications across various areas, including constrained optimization~\citep{beck2009fast}, machine learning~\citep{Bottou2018Optimization,lan2020first}, and operations management~\citep{porteus2002foundations,Chen2013Sparse,Wang2023Large}. The problem is characterized as a nonsmooth stochastic optimization problem, where the objective function consists of the summation of a smooth component (typically defined as the expectation of a stochastic function) and a nonsmooth component (typically defined as a regularization term); see Problem~\eqref{problem} for formal definition. 
In practice, the true gradient of the smooth term is often unavailable, and we only have access to observations of the stochastic function, which leads to an estimated gradient. 

Proximal gradient method is a classical algorithm for addressing nonsmooth optimization problems, consisting of a gradient step followed by a proximal mapping~\citep{Beck2017First}. The acceleration method is another widely used technique aimed at enhancing the efficiency of optimization algorithms~\citep[][]{polyak1964some,beck2009fast,toh2010accelerated,walker2011anderson,shen2011accelerated,Ghadimi2012Optimal,Yang2024Data}. Integrating these two concepts,  \cite{Nesterov2012Gradient,nesterov2013introductory} propose the accelerated proximal gradient method, referred to as NAPG (\textbf{N}esterov \textbf{A}ccelerated \textbf{P}roximal \textbf{G}radient), which effectively addresses nonsmooth optimization problems while achieving the optimal iteration complexity for first-order methods. However, NAPG is designed to handle optimization problems with known true gradients, and it becomes intricate and challenging when the true gradient is unknown. As discussed in \cite{assran2020convergence}, NAPG may fail to converge when the gradient is significantly inaccurate. This naturally leads to the research question: How can we effectively combine NAPG with sampling-based estimated gradients to tackle the stochastic composite optimization problem?

In the existing literature, several sampling strategies integrate gradient estimation with optimization algorithms. A prevalent method is \textit{deterministic sampling strategy} \citep{Friedlander2012Hybrid, Shanbhag2015Budget, Jalilzadeh2016eg, Jalilzadeh2022variable, Jalilzadeh2022Smoothed}, where the sample size sequence is predetermined and serves as input for an optimization algorithm. While these strategies have strong theoretical support, they frequently depend on prior knowledge of the problem characteristics, which may not be accessible in practical scenarios. For instance, the deterministic sampling strategies exhibit significant variations between convex and strongly convex optimization problems. As illustrated in \cite{lei2024variance}, achieving optimal complexity necessitates a geometric increase in sample size for strongly convex problems, whereas a polynomial growth rate is required for convex problems. Furthermore, an excessive number of samples in convex problems does not inherently accelerate the optimization algorithm. In practice, determining the convexity of a problem can be NP-hard~\citep{ahmadi2013np}. Even when convexity is established, deterministic sample size sequences still rely on parameters that are challenging to ascertain, such as the Lipschitz constant or the condition number of the problem. Inaccurate estimation or inappropriate specification of these parameters can significantly impair algorithm performance.

\textit{Adaptive sampling strategies}, which dynamically adjust sample sizes in accordance with algorithmic iterations, offer a more efficient alternative to deterministic sampling strategies. These adaptive sampling strategies commence with a small sample size and adaptively increase it as necessary to ensure that the optimization process converges at a desired rate. However, prior studies primarily focus on non-accelerated algorithms \citep{Xie2020Constrained, Bollapragada2023Adaptive} or assume smoothness in objectives \citep{Byrd2012Sample, Bollapragada2018Adaptive}, thereby leaving a critical gap: no existing framework integrates adaptive sampling strategies with NAPG for stochastic composite optimization problems.

In this paper, we address this gap by proposing the first adaptive sampling strategy tailored to NAPG for stochastic composite optimization problems, referred to as \textit{adaNAPG} (\textbf{ada}ptive \textbf{N}esterov \textbf{A}ccelerated \textbf{P}roximal \textbf{G}radient). Our approach obviates the necessity for prior parameter knowledge while achieving optimal complexities and facilitating statistical inference through a central limit theorem. Our contributions are summarized as follows:
\begin{enumerate} 
    \item[1.] 
    We propose a new practical framework for addressing stochastic optimization problems that integrates adaptive sampling strategies with NAPG, referred to as adaNAPG. To the best of our knowledge, it represents the first accelerated algorithm that adaptively determines sample sizes at each iteration without prior knowledge of the problem structure.
    \item[2.] 
    We provide a complexity analysis demonstrating that the new algorithm adaNAPG achieves both optimal iteration complexity and optimal sample complexity under appropriate conditions. To achieve an $\varepsilon$-suboptimality of the expected objective value, the iteration complexity of adaNAPG is, respectively, $O(1/\sqrt{\varepsilon})$ and $O(\log(1/\varepsilon))$ for convex and $\mu$-strongly convex objectives. It aligns with the best results of NAPG given by \cite{nesterov2013introductory}. Additionally, for $\mu$-strongly convex objectives, adaNAPG attains the best sample complexity proposed in \cite{Pasupathy2018sampling}.
    \item[3.] 
    We establish a central limit theorem (CLT) for the iteration sequence of the new algorithm adaNAPG, utilizing some techniques from nonsmooth analysis. The asymptotic result demonstrates that the convergence rate of the iteration sequence occurs at a geometric (exponential) rate, elucidating the efficiency of the new algorithm adaNAPG.
\end{enumerate}

The remainder of this paper is structured as follows. In Section~\ref{sec:Problem Definition}, we formulate the problem and introduce our sampling strategy and algorithm. Sections~\ref{sec:Convergence Rate Analysis} and~\ref{sec:asymptotic} present the complexity analysis and asymptotic behavior analysis of the algorithm respectively. Section~\ref{sec:conclusions} concludes the paper. All technique proofs are collected in Appendixes~\ref{app:Convergence Rate Analysis} and~\ref{app:asymptotic}. The e-companion Appendixes~\ref{EC:Auxiliary Nonsmooth Analysis Results}--\ref{EC:DCLT} list some lemmas used in the proofs, and Appendixes~\ref{ec:Numerical Experiments}--\ref{appendix:remark_clt_noise} provide numerical experiments demonstrating the effectiveness of our algorithm.

\section{Problem Formulation}\label{sec:Problem Definition}

We consider the stochastic composite optimization problem~\citep[][]{ porteus2002foundations,nesterov2013introductory,Beck2017First,Bottou2018Optimization,lan2020first,Wang2023Large}:
\begin{equation}
  \min_{x \in \mbR^{d}} \; \{F(x) := f(x) + h(x)\},
  \label{problem}
\end{equation}
where $f(x)$ is smooth and $h(x)$ is nonsmooth. Typically, $f(x):=\mbE[\tilde{f}(x;\xi)]$ signifies the expectation of a stochastic function $\tilde{f}(x;\xi)$ within a specified probability space, contingent upon a random variable $\xi$. 

\begin{assumption}\label{assumption_problem}
	Suppose that Problem~\eqref{problem} has the following properties: 
  \begin{enumerate}
    \item[(i)] $f \in \mcS_{\mu,L}$ for some $\mu\geq 0$ and $L>0$, where $\mcS_{\mu,L}$ denotes the class of $\mu$-strongly convex functions with $L$-Lipschitz gradient;
    \item[(ii)] $h(x):\mbR^{d}\to\mbR \cup \{+\infty\}$ is proper, closed and convex (but not necessarily differentiable);
    \item[(iii)] Problem~\eqref{problem} is bounded below and its optimal set, denoted by $X\opt$, is nonempty. 
  \end{enumerate}
  \end{assumption}
  
Assumption~\ref{assumption_problem} is standard in the convex optimization literature~\citep[][]{Beck2017First,Schmidt2017Minimizing,Yang2024Data}. Formally, the family of classes $\mcS_{\mu,L}$ also contains the class of convex functions with Lipschitz gradient by setting $\mu=0$. We now introduce the gradient estimation process and some related notations. For $x \in \mbR^{d}$, let $g(x,\xi)$ denotes the gradient estimation for $\nabla f(x)$ with single random sample $\xi$. To reduce variance, a batch of samples $\{g(x,\xi^{k}),~k=1,\dots,K\}$, where $K$ represents the sample size, may be employed. Consequently, we obtain the following gradient estimator  $ \hnabla f(x)=\frac{1}{K}\sum_{k=1}^{K}g(x,\xi^{k})$. For a vector $x$, we use $\|x\|$ to represent its Euclidean norm throughout the paper.  

\begin{assumption}\label{assumption_gradient}
The gradient estimator $g(x,\xi)$ is unbiased and has finite variance, that is, for all $x\in\mbR^{d}$, it holds that $\mbE[g(x,\xi)] = \nabla f(x)$ and $\mbE[\norm{g(x,\xi)-\nabla f(x)}^{2}] < \infty$. 
\end{assumption}

Assumption~\ref{assumption_gradient} is a common assumption. The unbiasedness of the gradient estimator can be obtained through various methods \citep[][]{Fu2015Handbook,mohamed2020monte}. The finite variance in Assumption~\ref{assumption_gradient} simply requires that the gradient estimation possess finite variance for all $x$, without imposing a uniform bound or any specific structural constraints. 

To capitalize on the inherent smoothness of $f(x)$ in Problem~\eqref{problem}, we focus on the Nestorov accelerated proximal gradient (NAPG) method, whose iterative procedure is as follows,
\begin{equation}\label{NAPG}
x_{n+1}   =\prox_{\alpha_{n}}(y_n - \alpha_{n} \nabla f(y_n)),\quad
y_{n+1}   =x_{n+1} + \beta_{n}(x_{n+1}-x_n).
\end{equation}
Here $\alpha_{n}$ and $\beta_n$ represent sequences of step sizes and the proximal operator is defined as
\begin{equation*}
  \prox_{h}(x) := \argmin_{u \in \mbR^{d}}\left\{ h(u)+ \frac{1}{2}\norm{u-x}^{2}  \right\}. 
  \label{proximal_operator}
\end{equation*}
The NAPG~\eqref{NAPG} incorporates a form of ``momentum" from previous iterations to achieve acceleration, and it can attain the optimal convergence rate of first-order methods when the true gradient $\nabla f(y_n)$ in Equation~\eqref{NAPG} is available~\citep{nesterov2013introductory}. Nonetheless, when the true gradient $\nabla f(y_n)$ is unavailable, \cite{konevcny2015mini, allen2018katyusha, assran2020convergence} indicate that Equation~\eqref{NAPG}, which utilizes the gradient $\hnabla f(y_{n})$ estimated by a \textit{deterministic sample size}, might fail to converge. It leads to the following research question: How to design an \textit{adaptive sampling scheme} ensuring the convergence of Equation~\eqref{NAPG} with the gradient $\hnabla f(y_{n})$ (estimated by the adaptive sample size), while preserving the optimal convergence rate?

\subsection{Adaptive Sampling Scheme}

Before giving our sampling strategy, we first introduce the concept of the \textit{gradient mapping}, which represents the optimal condition for the composite optimization problem~\citep{Beck2017First} in the sense that for any $x\opt\in \mathrm{int}(\mathrm{dom}(f))$ and $\alpha>0$, it holds that $G_{\alpha}(x\opt)=0$ if and only if $x\opt$ is an optimal solution to Problem~\eqref{problem} when  Assumption~\ref{assumption_problem} holds. The gradient mapping also plays a key role in our sampling scheme design.
\begin{definition}[gradient mapping]\label{definition:gradient_mapping}
  The gradient mapping is the operator defined by
  \begin{equation}\label{eq:gradient mapping}
    G_{\alpha}(x):= \frac{1}{\alpha}\left[x-\prox_{\alpha h}\left(x-\alpha \nabla f(x)\right)\right].
  \end{equation}
\end{definition}
When replacing $\nabla f(x)$ with $\hnabla f(x)$ in Definition~\ref{definition:gradient_mapping}, we obtain the \textit{sample gradient mapping}
  \begin{equation}\label{eq:sample gradient mapping}
    \hat{G}_{\alpha}(x):= \frac{1}{\alpha}\left[x-\prox_{\alpha h}\left(x-\alpha\hnabla f(x)\right)\right].
  \end{equation}

We present an adaptive sample allocation strategy for Equation~\eqref{NAPG}, where the gradient $\nabla f(x)$ is estimated using the allocated samples at each iteration. By taking a close look at Equation~\eqref{NAPG}, the gradient estimation $\hnabla f(\cdot)$ is only required at each $y_{n}$. Then, for each $y_{n}$, let $K_{n}$ denote the sample size, and $\{\xi^{i}_{n}, 1\leq i \leq K_{n}\}$ denote the generated i.i.d. samples. 
Let $\mcF_{n}:=\sigma\{x_{0}, \xi^{i}_{t}, 1\leq i \leq K_{t}, 0\leq t \leq n-1 \}$ represent the filtration generated by the first $n-1$ iteration, and we denote $\mbE_{n}[\cdot]:=\mbE[\cdot \mid \mcF_{n}]$ for all $n\geq 1$. For each iterator $y_{n}$, we require the sample size $K_n$  ensuring that the gradient estimation $\hnabla f(y_{n})$ satisfies the following two tests almost surely for some $\theta,\nu>0$: 
\begin{align}
  \mbE_{n}\left[\left(  \frac{\hnabla f(y_{n})^{\top} \nabla f(y_{n})}{\norm{\nabla f(y_{n})}} - \norm{\nabla f(y_{n})}\right)^{2} \right] \leq \theta^{2} \norm{G_{\alpha_{n}}(y_{n})}^{2},
  \label{NAPG_test1}\\
  \mbE_{n}\left[ \left\|\hnabla f(y_{n}) - \frac{\hnabla f(y_{n})^{\top} \nabla f(y_{n})}{\norm{\nabla f(y_{n})}^{2}} \nabla f(y_{n})\right\|^{2} \right] \leq \nu^{2}\norm{G_{\alpha_{n}}(y_{n})}^{2}.
  \label{NAPG_test2}
\end{align}
We define Equations~\eqref{NAPG_test1}--\eqref{NAPG_test2} as the sample tests for sample allocation at the $n$-th iteration. This adaptive sampling strategy is inspired by \cite{Bollapragada2018Adaptive}. Notice that, when $h(x)\equiv 0$, these two tests~\eqref{NAPG_test1}--\eqref{NAPG_test2} reduce to the inner-product test and the orthogonality test outlined in \cite{Bollapragada2018Adaptive}. In practical applications, obtaining the gradient $\nabla f(y_{n})$, the gradient mapping $G_{\alpha_n}(y_n)$, and the expectation $\mathbb{E}_{n}$ as required in Tests~\eqref{NAPG_test1}--\eqref{NAPG_test2} directly may not be feasible. Hence, we resort to using the estimated gradient $\hat{\nabla} f(y_{n})$, the sample gradient mapping $\hat{G}_{\alpha_n}(y_n)$, and the sample mean. This heuristic approach is employed by \cite{Bollapragada2018Adaptive} to address Problem~\eqref{problem} with $h(x)\equiv 0$. For notational simplicity, in the rest of this paper we define $\tilde{w}_{n}:=\hnabla f(y_n)-\nabla f(y_n)$ as the noise of the gradient estimation at point $y_n$ and $w_{n}:=\hat{G}_{\alpha_{n}}(y_{n})-G_{\alpha_{n}}(y_{n})$ as the gradient mapping error for each step.

Figure~\ref{fig:NAPG_test} depicts a geometric interpretation of Tests~\eqref{NAPG_test1}--\eqref{NAPG_test2}, elucidating their role in reducing variance. Figure~\ref{fig:NAPG_test1} delineates the feasible range of the gradient estimation $\hnabla f(y_{n})$ determined by Test \eqref{NAPG_test1}, where $\hnabla f(y_{n})^{\top} \nabla f(y_{n}) /\norm{\nabla f(y_{n})}$ represents the ``length" of the projection of $\hat{\nabla} f(y_{n})$ onto the true gradient $\nabla f(y_{n})$. The shaded region in Figure~\ref{fig:NAPG_test2} delineates the feasible range of the gradient estimation $\hnabla f(y_{n})$ determined by Test~\eqref{NAPG_test2}, which precisely governs this projection ``residual" (represented by the dashed lines) and prevents $\hat{\nabla} f(y_{n})$ from becoming orthogonal to $\nabla f(y_{n})$. Simultaneously satisfying Tests \eqref{NAPG_test1} and \eqref{NAPG_test2} with the gradient estimation $\hat{\nabla} f(y_{n})$ could effectively reduce variance. Intuitively, as the iteration sequence $\{y_{n}\}$ approaches the optimal solution, $G_{\alpha_{n}}(y_{n})$ tends to zero, causing the feasible region constrained by Tests~\eqref{NAPG_test1} and \eqref{NAPG_test2} for the gradient estimator to shrink. Consequently, the gradient estimator $\hat{\nabla} f(y_{n})$ achieves accuracy by closely aligning with the true gradient $\nabla f(y_{n})$.

\begin{figure}
  \FIGURE
  {
  \subcaptionbox{Diagram of Test~\eqref{NAPG_test1}\label{fig:NAPG_test1}} 
  {\includegraphics[width=0.4\textwidth]{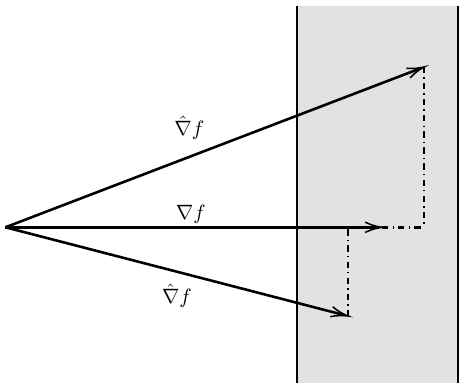}} 
  \hfill\subcaptionbox{Diagram of Test~\eqref{NAPG_test2}\label{fig:NAPG_test2}} {\includegraphics[width=0.4\textwidth]{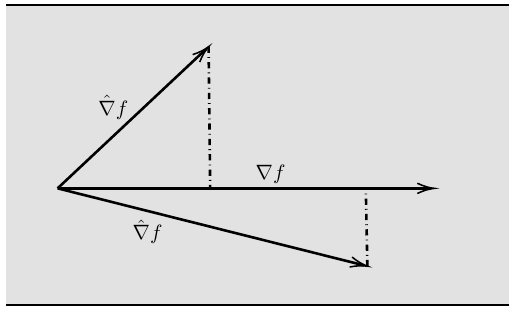}} 
  } 
  {Geometric interpretation of Test~\eqref{NAPG_test1} and Test~\eqref{NAPG_test2}. \label{fig:NAPG_test}}{}
\end{figure}

Algorithm~\ref{algorithm_adaNAPG} presents our new algorithm \textit{adaNAPG} (\textbf{ada}ptive \textbf{N}esterov \textbf{A}ccelerated \textbf{P}roxiaml \textbf{G}radient method) designed for Problem~\eqref{problem}. Each iteration involves the generation of specific gradient samples to satisfy Tests~\eqref{NAPG_test1}--\eqref{NAPG_test2} concurrently. Subsequently, we utilize the estimator $\hat{\nabla} f(y_{n})$ to drive the NAPG~\eqref{NAPG}. Notice that, for the case where the smooth term $f$ in Problem~\eqref{problem} exhibits strong convexity with a parameter $\mu>0$, the iteration scheme in Algorithm~\ref{algorithm_adaNAPG} can be simplified. Specifically, by selecting $\pi_{0}=\sqrt{\mu[L(\theta^{2}+\nu^{2}+1)]^{-1}}=\sqrt{\mu\alpha}$, it can be readily verified that $\pi_{n}=\sqrt{\mu\alpha}$ for all $n$. Consequently, the term $\pi_{n}$ can be eliminated, leading to a simplified iteration rule for $y_n$ in Algorithm~\ref{algorithm_adaNAPG} as $y_{n+1} = x_{n+1} + \beta(x_{n+1}-x_{n})$, where $\beta=(1-\sqrt{\alpha\mu})/(1+\sqrt{\alpha\mu})$.

\begin{algorithm}
  \caption{Adaptive sampling proximal gradient method with Nesterov acceleration}  \label{algorithm_adaNAPG}
  \KwIn{Initial iterator $x_{0}$; Sample test parameters: $\theta>0$, $\nu>0$;  $\pi_{0}\in(0,1)$, $\mu$, $L$;}
  \KwOut{$\{x_n\}_{n\geq 1}$;}
  Set $y_{0}\leftarrow x_{0}$, $n\leftarrow 0$,  $\alpha_n \equiv \alpha :=[L(\theta^{2}+\nu^{2}+1)]^{-1}$, $q=\mu\alpha$\;
  \While{convergence condition is not satisfied}{Generate a gradient estimator $\hnabla f(y_{n})$ satisfying both Test~\eqref{NAPG_test1} and Test~\eqref{NAPG_test2}\;
  Compute $x_{n+1}=\prox_{\alpha h}(y_{n}-\alpha \hnabla f(y_{n}))$\;
  Compute $\pi_{n+1}$ from the equation $\pi_{n+1}^{2} - (q-\pi_{n}^{2})\pi_{n+1} - \pi_{n}^{2} = 0$\;
  Compute $y_{n+1}=x_{n+1}+\frac{\pi_{n}(1-\pi_{n})}{\pi_{n}^{2}+\pi_{n+1}}(x_{n+1}-x_{n})$\;
  Set $n\leftarrow n+1$\;
  }
  \end{algorithm}

\section{Complexity Analysis}\label{sec:Convergence Rate Analysis}

This section establishes the \textit{iteration complexity} and the \textit{sample complexity} of Algorithm~\ref{algorithm_adaNAPG} respectively. The iteration complexity refers to the number of algorithmic iteration needed such that $\mbE[F(x_n)]-F(x\opt)$ is small enough, whereas the sample complexity refers to the amount of samples needed such that $\norm{x_n-x\opt}^2$ is small enough. Theorem~\ref{adaNAPG_convergence} provides the iteration complexity of Algorithm~\ref{algorithm_adaNAPG} for the convex setting and the strongly convex setting respectively.

\begin{theorem}[Iteration Complexity]\label{adaNAPG_convergence}
  Suppose that Assumptions~\ref{assumption_problem} and~\ref{assumption_gradient} hold. Let $\{x_{n}\}$ be the sequence generated by Algorithm~\ref{algorithm_adaNAPG} and $x\opt \in X\opt$. If $f \in \mathcal{S}_{0,L}$, then we have 
  \begin{equation*}
        \mbE[F(x_{n})]-F(x\opt) \leq \frac{4(\theta^{2}+\nu^{2}+1)}{\big( 2\sqrt{\theta^{2}+\nu^{2}+1}+n \big)^{2}}C_{1},
    \end{equation*}
  where $C_{1} = F(x_{0})-F(x\opt)+\frac{L}{2}\norm{x_{0}-x\opt}^{2}$. If $f\in \mathcal{S}_{\mu,L}$, for some $\mu>0$, then we have 
    \begin{equation*}
        \mbE[F(x_{n})]-F(x\opt) \leq  C_{2}\rho^{n} \quad \text{and} \quad \mbE\left[\norm{x_{n}-x\opt}^{2}\right] \leq C_{2}\rho^{n}.
    \end{equation*}
    where  $\rho=1-\sqrt{\mu[L(\theta^{2}+\nu^{2}+1)]^{-1}}< 1$ and $C_{2}=\frac{2}{\mu}[F(x_{0})-F(x\opt)+\frac{\mu}2 \norm{x_{0}-x\opt}^{2}]$. 
  \end{theorem}
  
As pointed out by \cite{allen2018katyusha}, Equation~\eqref{NAPG} may fail to converge when replacing $\nabla f(y_n)$ by a general gradient estimator $\hnabla f(y_n)$ without any sampling scheme to control its variance. Theorem~\ref{adaNAPG_convergence} ensures the convergence of Equation~\eqref{NAPG} since we involve sampling tests~\eqref{NAPG_test1}--\eqref{NAPG_test2} to control variance of $\hnabla f(y_n)$. It further shows that the iteration complexity is $O(1/\sqrt{\varepsilon})$ for the convex setting and $O(\log(1/\varepsilon))$ for the strongly convex setting, which achieve the lower bounds for first-order methods~\citep{nesterov2013introductory}. 

In the following we delineates the sample complexity of Algorithm~\ref{algorithm_adaNAPG}. Let $\Gamma_{n}$ be the total sample size utilized in the first $n$ iterations. \cite{Pasupathy2018sampling} prove that the lower bound of sample complexity for Equation~\eqref{NAPG} is of $\Gamma_{n}^{-1}$, when replacing $\nabla f(y_n)$ by a general gradient estimator. \cite{Pasupathy2018sampling} further show that this lower bound can only be achieved when the recursion enjoys a linear rate of convergence (in our framework it reduces to $f(x)$ being strongly convex). 
 
\begin{theorem}[Sample Complexity]\label{adaNAPG_simulation} 
      Suppose that Assumption~\ref{assumption_problem} holds with $\mu>0$. Further assume that the gradient estimator $g(x,\xi)$ has uniformly bounded variance, that is, for all $x\in\mbR^{d}$, we have $\mbE[\norm{g(x,\xi)-\nabla f(x)}^{2}] \leq \sigma^{2}$, where $\sigma>0$ is a constant. Then, it holds that $\norm{x_{n}-x\opt}^{2} = O_{\mathbb{P}}(\Gamma_{n}^{-1})$.
    \end{theorem}
    
Theorem~\ref{adaNAPG_simulation} shows that the convergence rate of $\norm{x_{n}-x\opt}^{2}$ is $O_{\mathbb{P}}(\Gamma_{n}^{-1})$, i.e., for any $\delta>0$, there exists $c>0$ such that $\Pr{\Gamma_n \cdot \norm{x_{n}-x\opt}^{2} > c} \leq \delta$ for all $n$. This result demonstrates that Algorithm~\ref{algorithm_adaNAPG} attains the lower bound given in~\cite{Pasupathy2018sampling}. The existing literature on adaptive sampling strategies mostly focus on the iteration complexity~\citep[][]{Byrd2012Sample,Hashemi2014adaptive,Bollapragada2018Adaptive,Franchini2023line}, and they do not carefully study the sample complexity. 

\section{Asymptotic Behavior Analysis}\label{sec:asymptotic}

In this section, we examine the asymptotic behavior of the iteration sequence from Algorithm~\ref{algorithm_adaNAPG}. \cite{lei2024variance} show that for smooth problems ($h(x)=0$ in Problem~\eqref{problem}), the iteration sequences of three classical algorithms (SGD, heavy ball method, accelerated SGD) converge in distribution to a normal distribution with proper deterministic sample size sequence. However, they focus on smooth problems without adaptive sampling strategies. The methods in \cite{lei2024variance} can not be directly extended to scenarios where $h(x) \neq 0$, primarily because the gradient mapping $G_{\alpha}(\cdot)$ exhibits nonlinearity or nonsmoothness. 

\begin{assumption}\label{assum_G_semismooth}
   Suppose that the gradient mapping $G_{\alpha}(x)$ is semismooth at $x\opt \in X\opt$ and the directional derivative of $G_{\alpha}(x)$ exists for any direction at $x \in \mbR^{d}$.
\end{assumption}

The concept of semismooth in Assumption~\ref{assum_G_semismooth} is basic in the literature of nonsmooth optimization \citep{Facchinei2003Finite}; see Appendix~\ref{EC:Auxiliary Nonsmooth Analysis Results} for formal definition. Generally, the gradient mapping in many applications are shown to be  semismooth~\citep[][]{sun2002semismooth,Facchinei2003Finite,rockafellar2009variational,deng2025augmented}. The existence of the directional derivative is a mild condition and weaker than the differentiability. 

\begin{assumption}\label{assum_f_sc}
Suppose that the smooth term $f$ in Problem~\eqref{problem} satisfies  $f \in \mathcal{S}_{\mu,L}$ for some $\mu>0$ and twice differentiable. That is, for all $x \in \mbR^{d}$, it holds that $\mu\mathbf{I} \prec \nabla^{2} f(x) \prec L\mathbf{I}$, where $\mathbf{I}$ is the identity matrix.
\end{assumption}

Here for two matrices $A$ and $B$, $A\prec B$ means $B-A$ is positive semi-definite. The twice differentiability of $f$ in Assumption~\ref{assum_f_sc} is a standard assumption when studying the CLT result of stochastic approximation algorithms \citep[][]{pflug1990non,polyak1992acceleration,hu2024quantile}. Under Assumptions~\ref{assum_G_semismooth}--\ref{assum_f_sc} we can establish a linear recursive relationship for Algorithm~\ref{algorithm_adaNAPG}.

\begin{lemma}\label{lemma_linear_recursion}
 Consider the iteration sequence $\{x_{n}\}$ in Algorithm~\ref{algorithm_adaNAPG}.  
Suppose that Assumptions~\ref{assumption_problem} and~\ref{assum_G_semismooth}--\ref{assum_f_sc} hold. There exist a sequence of matrices $\{H_{n}\}$ and a sequence of vectors $\{\delta_{n}\}$ such that
\begin{equation}\label{eq:e_n}
    e_{n+1} = P_{n}e_{n} - \alpha\rho^{-\frac{n+1}2}\begin{pmatrix}
      w_{n} \\ 0
    \end{pmatrix} + \begin{pmatrix} \zeta_{n}\\0 \end{pmatrix},
\end{equation}
where the constant $\rho$ is defined in Theorem~\ref{adaNAPG_convergence}, $w_n=\hat{G}_{\alpha}(y_{n})-G_{\alpha}(y_{n})$ defined in Section~\ref{sec:Problem Definition}, and
\begin{equation*}
    e_{n} :=\rho^{-n/2}\begin{pmatrix}
  x_{n}- x\opt \\ x_{n-1} - x\opt
\end{pmatrix},\quad
P_{n}  :=\rho^{-1/2}\begin{pmatrix}
  (1+\beta)(\mathbf{I}-\alpha H_{n}) & -\beta(\mathbf{I}-\alpha H_{n})\\
  \mathbf{I} & 0
\end{pmatrix},\quad\zeta_{n} :=-\alpha\rho^{-(n+1)/2} \delta_{n}.
\end{equation*}
Here we take $x_{-1}=x^*$ for notation.  
\end{lemma}

The linear recursion~\eqref{eq:e_n} further implies the following decomposition of $e_n$,
\begin{equation}\label{eq:e_n_decom}
    e_n = \prod_{i=0}^{n}P_i e_{0} -\alpha \sum_{i=0}^{n} \prod_{t=i+1}^{n}P_t \rho^{-\frac{t+1}2}\begin{pmatrix}
      w_{t} \\ 0
    \end{pmatrix} + \sum_{i=0}^{n} \prod_{t=i+1}^{n}P_t \begin{pmatrix} \zeta_{t}\\0 \end{pmatrix}.
\end{equation}
The sequence $\{P_n\}$ is particularly important as it dictates how the sequence $\{e_n\}$ are transformed and projected at each iteration, effectively characterizing the dynamics of the optimization process. The recursions~\eqref{eq:e_n} and~\eqref{eq:e_n_decom} decompose the sequence~$\{e_n\}$ into two key components: One related to the gradient mapping estimation error $\{w_{n}\}$, and another to a linear approximation error $\{\zeta_n\}$ (see Lemma~\ref{lemma_G_Jacobian} in Appendix~\ref{app:asymptotic} for the linear approximation). Next, we make assumptions for the sequences $\{P_{n}\}$ and $\{\tilde{w}_n\}$ respectively.

\begin{assumption}\label{assum_spectral radius_Pn}
    Suppose that there exits some $0<\kappa<1 $ such that $\varrho(P_{n})< \kappa$ for all $n$, where $\varrho(P_{n})$ is the spectral radius of the matrix $P_{n}$.
\end{assumption}

Under Assumption~\ref{assum_spectral radius_Pn} the matrix $P_{n}$ acts as a contraction, reducing the effect of the initial point $x_{0}$ and ensuring that the cumulative error in Equation~\eqref{eq:e_n_decom} does not escalate as the algorithm converges. Assumption~\ref{assum_spectral radius_Pn} holds under some usual cases. When Problem~\eqref{problem} degenerates to the smooth optimization, that is $h(x)=0$, Assumption~\ref{assum_spectral radius_Pn} will hold naturally. Another sufficient condition for Assumption~\ref{assum_spectral radius_Pn} is that the eigenvalues of $H_n$ are all real and positive and are bounded by $\mu$ and $L$, which can be verified following the similar discussion of Equation (42) in \cite{lei2024variance} and thus omit here. 

\begin{assumption}\label{assum_clt_noise} Let $\tilde{w}_{n}=\hnabla f(y_n)-\nabla f(y_n)$ defined in Section~\ref{sec:Problem Definition}. Then, 
    \begin{enumerate}
        \item[(i)] There exists a semi-definite matrix $\tilde{S}_{0}$ such that 
        \begin{equation}\label{eq51}
            \lim_{n \to \infty}  \rho^{-n} \mbE\left[\tilde{w}_{n}\tilde{w}_{n}^{\top}\right]  = \tilde{S}_{0},
        \end{equation}
        where $\rho \in (0,1)$ is the geometric convergence rate from Theorem~\ref{adaNAPG_convergence}.
        \item[(ii)] The following Lindeberg's condition holds.
        \begin{equation}\label{eq52}
            \lim_{r \to \infty}\sup_{n}\mbE \left[\norm{\rho^{-n/2}\tilde{w}_{n}}^{2}\mathbb{I}_{[\norm{\rho^{-n/2}\tilde{w}_{n}}>r]} \right] = 0,
        \end{equation}
        where $\mathbb{I}_{C}$ is the indicator function of event $C$.
    \end{enumerate}
    \end{assumption}

The condition (i) of Assumption~\ref{assum_clt_noise} is a standard regular condition for the noise sequence.  We provide a detailed numerical evidence in Section~\ref{appendix:remark_clt_noise} to support this condition. The condition (ii) of Assumption~\ref{assum_clt_noise} is the Lindeberg's condition, which is a standard condition for the CLT-type results~\citep{VanderVaart2000Asymptotic}. 

With Assumptions~\ref{assum_spectral radius_Pn} and~\ref{assum_clt_noise}, we can show that the sequence $\{e_n\}$ in Equation~\eqref{eq:e_n} follows an asymptotic normal distribution, which gives the limiting behavior of the iteration sequence in Algorithm~\ref{algorithm_adaNAPG}.

\begin{theorem}\label{CLT_adaNAPG}
    Suppose that Assumptions~\ref{assumption_problem}--\ref{assum_clt_noise} hold. Then, there exists a positive semi-definite matrix $\Sigma$ such that the sequence $\{x_{n}\}$ generated by Algorithm~\ref{algorithm_adaNAPG} satisfies
    \begin{equation*}
      \alpha^{-1}\rho^{-n/2}\begin{pmatrix}
      x_{n}- x\opt \\ x_{n-1} - x\opt
    \end{pmatrix}\Rightarrow \mathrm{Normal}(0,\Sigma) \quad \mbox{as $n\to\infty$},
    \end{equation*}
     where $\rho$ is defined in Theorem~\ref{adaNAPG_convergence}.
    \end{theorem}
    
Theorem~\ref{CLT_adaNAPG} demonstrates that the sequence $\{x_{n}\}$ generated by Algorithm~\ref{algorithm_adaNAPG} converges in distribution to a normal distribution at a geometric rate. Building upon Theorem~\ref{CLT_adaNAPG}, we can construct confidence regions for the optimal solution $x^{\ast}$ using the sample covariance matrix, which is effective when the dimension of $x$ is small. Nonetheless, as the dimension of $x$ increases, the efficient construction of a confidence region remains challenging. The development of efficient confidence regions for high-dimensional optimization algorithms is a noted difficulty in the literature \citep[][]{hsieh2002confidence,zhu2021constructing}, and we leave this for future endeavor.

\section{Conclusions}\label{sec:conclusions}

In this paper, we propose adaptive sampling algorithms with stochastic accelerated proximal gradient methods for addressing the stochastic composite problem. We establish a complexity analysis for the proposed algorithm. For strongly convex problems, we provide a central limit theorem for the iteration sequence generated by the new algorithm. Our research can be expanded in several directions. For instance, developing adaptive sampling methods for biased gradient estimators, which frequently arise in Zeroth-order optimization~\citep{hu2024convergence} and Markov chain gradient descent~\citep{sun2018markov}, would be valuable. 

\begin{APPENDICES}
\section{Proofs in Section~\ref{sec:Convergence Rate Analysis}}\label{app:Convergence Rate Analysis}

In this section, we prove Theorems~\ref{adaNAPG_convergence} and~\ref{adaNAPG_simulation} in Section~\ref{sec:Convergence Rate Analysis}. Before that, we prove some auxiliary lemmas.

\subsection{Auxiliary Lemmas}

\begin{lemma}\label{lemma:adaNAPG sufficient descent}
Suppose that Assumptions~\ref{assumption_problem} and~\ref{assumption_gradient} hold. Then, for any $z \in \mathcal{F}_{n}$, the iteration sequences $\{x_{n}\}$ and $\{y_{n}\}$ in Algorithm~\ref{algorithm_adaNAPG} satisfy
  \begin{equation}\label{eq83}
    \mbE_{n}[F(x_{n+1})]  \leq \mbE_{n} \left[ F(z) -\frac{\mu}{2}\bignorm{y_{n}-z}^{2} -\frac{1}{L} \bignorm{G_{\alpha}(y_{n})}^{2}-\iproduct{\hat{G}_{\alpha}(y_{n})}{z-y_{n}} \right].
 \end{equation}
\end{lemma}
\begin{proof}{Proof of Lemma~\ref{lemma:adaNAPG sufficient descent}}
    Since $f \in \mcS_{\mu,L}$ and $x_{n+1}=y_{n}-\alpha \hat{G}_{\alpha}(y_{n})$, we have~\citep{Ghadimi2016Accelerated}
\begin{equation}\label{eq:1}
  f(x_{n+1}) \leq f(y_{n})-\alpha \iproduct{\nabla f(y_{n})}{\hat{G}_{\alpha}(y_{n})}+\frac{L\alpha^{2}}{2}\bignorm{\hat{G}_{\alpha}(y_{n})}^{2}.
\end{equation}
From the definition of $\mu$-strong convexity, for any $z$ it holds that
\begin{equation}\label{eq:2}
f(z) \geq f(y_{n}) +\iproduct{\nabla f(y_{n})}{z-y_{n}} + \frac{\mu}{2}\bignorm{z-y_{n}}^{2}.
\end{equation}
By $x_{n+1}=\prox_{\alpha h}(y_{n}-\alpha \hnabla f(y_{n}))$ in the Line 4 of Algorithm~\ref{algorithm_adaNAPG}, Theorem 6.39  in~\cite{Beck2017First} gives us that 
$y_{n}-\alpha \hnabla f(y_{n}) - x_{n+1} \in \partial (\alpha h)(x_{n+1})$, 
which is
$\hG_{\alpha}(y_{n})-\hnabla f(y_{n})\in \partial  h(x_{n+1})$. Since $h$ is a convex function, for any $z$ we further have
\begin{equation}\label{eq:3}
    h(z)\geq  h(x_{n+1}) + \iproduct{\hG_{\alpha}(y_{n})-\hnabla f(y_{n})}{z-x_{n+1}}.
\end{equation}
Recalling $w_n= \hat{G}_{\alpha}(y_n) - G_{\alpha}(y_n)$ and $\tilde{w}_n = \hnabla f(y_{n}) - \nabla f(y_{n})$, then it holds that
    \begin{align}
    \nonumber  F(x_{n+1}) = & {} f(x_{n+1}) + h(x_{n+1})\\
   \nonumber   \overset{\eqref{eq:1}}{\leq}  & {}  f(y_{n})-\alpha \iproduct{\nabla f(y_{n})}{\hat{G}_{\alpha}(y_{n})}+\frac{L\alpha^{2}}{2}\bignorm{\hat{G}_{\alpha}(y_{n})}^{2} + h(x_{n+1})\\ \nonumber
      \overset{\eqref{eq:2}-\eqref{eq:3}}{\leq} & {} f(z)-\iproduct{\nabla f(y_{n})}{z-y_{n}} -\frac{\mu}{2}\bignorm{z-y_{n}}^{2} -\alpha \iproduct{\nabla f(y_{n})}{\hat{G}_{\alpha}(y_{n})}+\frac{L\alpha^{2}}{2}\bignorm{\hat{G}_{\alpha}(y_{n})}^{2}\\ \nonumber
      & + h(z) - \iproduct{\hat{G}_{\alpha}(y_{n})-\hnabla f(y_{n})}{z-y_{n}+\alpha \hat{G}_{\alpha}(y_{n})}\\ \nonumber
      = & {} F(z) -\frac{\mu}{2}\bignorm{z-y_{n}}^{2}+ \left(\frac{\alpha^{2}L}{2}-\alpha\right) \bignorm{\hat{G}_{\alpha}(y_{n})}^{2}+\iproduct{\tilde{w}_{n}}{z-y_{n}} \\ \nonumber
       & -\iproduct{\hat{G}_{\alpha}(y_{n})}{z-y_{n}} + \alpha \iproduct{\tilde{w}_{n}}{\hat{G}_{\alpha}(y_{n})}\\ \nonumber
      = & {} F(z) -\frac{\mu}{2}\bignorm{z-y_{n}}^{2} + \iproduct{\tilde{w}_{n}}{z-y_{n}+\alpha G_{\alpha}(y_{n})}-\iproduct{\hat{G}_{\alpha}(y_{n})}{z-y_{n}} \\ 
        & + \left(\frac{\alpha^{2}L}{2}-\alpha\right) \bignorm{\hat{G}_{\alpha}(y_{n})}^{2} + \alpha \iproduct{\tilde{w}_{n}}{w_{n}}.\label{eq70}
    \end{align}

From Test~\eqref{NAPG_test2} we know that 
  \begin{align}
   \nonumber & \mbE_{n}\left[ \bignorm{\hnabla f(y_{n}) - \frac{\hnabla f(y_{n})^{\top} \nabla f(y_{n})}{\bignorm{\nabla f(y_{n})}^{2}} \nabla f(y_{n})}^{2} \right] \\
    = & \mbE_{n}\left[ \bignorm{\hnabla f(y_{n})}^{2} \right] - \mbE_{n}\left[ \frac{(\hnabla f(y_{n})^{\top} \nabla f(y_{n}))^{2}}{\bignorm{\nabla f(y_{n})}^{2}} \right] \leq \nu^{2}\bignorm{G_{\alpha_{n}}(y_{n})}^{2}.\label{eq102}
  \end{align}
Similarly, from Test~\eqref{NAPG_test1} we can get
\begin{equation}\label{eq103}
     \mbE_{n}\left[\left(  \frac{\hnabla f(y_{n})^{\top} \nabla f(y_{n})}{\bignorm{\nabla F(y_{n})}} - \bignorm{\nabla F(y_{n})}\right)^{2} \right]
     \leq \theta^{2}\bignorm{G_{\alpha_{n}}(y_{n})}^{2}.
\end{equation}
Adding Equation~\eqref{eq102} and Equation~\eqref{eq103}, we have
\begin{equation}
 \mbE_{n}\left[ \bignorm{\tilde{w}_n}^{2} \right]= \mbE_{n}\left[ \bignorm{\hnabla f(y_{n})-\nabla f(y_{n})}^{2} \right] \leq (\theta^{2} + \nu^{2})\bignorm{G_{\alpha_{n}}(y_{n})}^{2}.
  \label{eq104}
\end{equation}
Since the proximal operator is Lipschitz ,for $w_n$  we have
\begin{equation*}
     \bignorm{w_n} =  \frac{1}{\alpha}\bignorm{\prox_{\alpha h}(y_n-\alpha\nabla f(y_n)) - \prox_{\alpha h}(y_n-\alpha(\nabla f(y_n)+\tilde{w}_n))} \leq \bignorm{\tilde{w}_n},
\end{equation*}
which implies that 
\begin{equation}\label{eq71}
  \mbE_{n}[\iproduct{\tilde{w}_n}{w_n}] \leq \mbE_{n}[\bignorm{\tilde{w}_n}\bignorm{w_n}]\leq \mbE_{n}[\bignorm{\tilde{w}_n}^{2}] \leq (\theta^{2}+\nu^{2}) \bignorm{G_{\alpha}(y_{n})}^{2}.
\end{equation}
Then we bound the term $\|\hat{G}_{\alpha}(y_{n})\|^{2}$. Since $\norm{a+b}^{2}\leq 2(\norm{a}^{2}+\norm{b}^{2})$, we have 
\begin{align}
\nonumber    \mbE_{n}\left[\bignorm{\hat{G}_{\alpha}(y_{n})}^{2}\right]&  = \mbE_{n}\left[\bignorm{w_n+G_{\alpha}(y_{n})}^{2}\right] \\
 \nonumber   & \leq 2 \mbE_{n}\left[\bignorm{w_n}^{2}+\bignorm{G_{\alpha}(y_{n})}^{2}\right]\\
    & \leq 2(\theta^{2}+\nu^{2}+1)\bignorm{G_{\alpha}(y_{n})}^{2}.\label{eq72}
  \end{align}
Taking expectation on both sides of Equation~\eqref{eq70} and taking Equations~\eqref{eq71} and \eqref{eq72} into account, then we prove the inequality 
\begin{equation*}
     \mbE_{n}[F(x_{n+1})]  \leq \mbE_{n}[F(z) -\frac{\mu}{2}\bignorm{y_{n}-z}^{2} -\frac{1}{L} \bignorm{G_{\alpha}(y_{n})}^{2}+\iproduct{\tilde{w}_{n}}{z-y_{n}+\alpha G_{\alpha}(y_{n})}-\iproduct{\hat{G}_{\alpha}(y_{n})}{z-y_{n}}],
  \end{equation*}
  where we also use $\alpha=[L(\theta^{2}+\nu^{2}+1)]^{-1}$. When $z \in\mathcal{F}_{n}$, Equation~\eqref{eq83} holds since 
\begin{equation*}
 \mbE_{n}[\iproduct{\tilde{w}_{n}}{z-y_{n}+\alpha G_{\alpha}(y_{n})}] = \iproduct{\mbE_{n}[\tilde{w}_{n}]}{z-y_{n}+\alpha G_{\alpha}(y_{n})}=0.\Halmos
\end{equation*}

\end{proof}

The convergence analysis of Algorithm~\ref{algorithm_adaNAPG} relies on the tools called \textit{estimate sequence} (see \citealt{nesterov2013introductory} for more details). Define the following sequences,
\begin{equation} \label{eq74}
  \phi_{0}(x) = F(x_{0})+\frac{\gamma_{0}}{2}\bignorm{x_{0}-x\opt}^{2},
\end{equation}
\begin{equation}\label{eq75}
  \phi_{n+1}(x) = (1-\pi_{n})\phi_{n}(x) + \pi_{n}\left(F(x_{n+1})-\eta \bignorm{G_{\alpha}(y_{n})}^{2} + \frac{\mu}{2}\bignorm{x-y_{n}}^{2}+\iproduct{\hat{G}_{\alpha}(y_{n})}{x-y_{n}}\right),
\end{equation}
Here $\pi_{n}$ is the positive root of the following quadratic equation,
\begin{equation}\label{eq82}
    L(\theta^{2}+\nu^{2}+1)\pi_{n}^{2}=(1-\pi_{n})\gamma_{n}+\pi_{n}\mu
  \end{equation}
and $\gamma_{n}$ is computed by the following recurrence,
\begin{equation}\label{eq85}
    \gamma_{n+1} =(1-\pi_{n})\gamma_{n}+\pi_{n}\mu
  \end{equation}
  for some initial value $\gamma_{0}>0$. One can verify that the sequence $\pi_{n}$ generated from Equation~\eqref{eq82} is the same as the one computed from Line 6 in Algorithm~\ref{algorithm_adaNAPG} by following the discussion about Equation (2.2.9) in \cite{nesterov2013introductory}. Next lemma shows the relationship between the estimate sequence and the function value sequence.

\begin{lemma}\label{lemma:estimate sequence}
    Suppose that Assumptions~\ref{assumption_problem} and~\ref{assumption_gradient} hold. Then we have 
    \begin{equation*}\label{eq87}
  \mbE[F(x_{n})] -F(x\opt) \leq \lambda_{n}[\phi_{0}(x\opt)-F(x\opt)], \quad \lambda_{n}=\prod_{i=1}^{n}(1-\pi_{i}).
\end{equation*}
\end{lemma}
\begin{proof}{Proof of Lemma~\ref{lemma:estimate sequence}}
 The proof is divided into five steps.
 \paragraph{Step 1} Firstly we show that $\nabla^{2} \phi_{n}(x)=\gamma_{n} \mathbf{I}$ by induction, where $\mathbf{I}$ is the identity matrix. By Equation~\eqref{eq74}, we have $\nabla^{2} \phi_{0}(x)=\gamma_{0} \mathbf{I}$. Suppose that $\nabla^{2} \phi_{n}(x)=\gamma_{n} \mathbf{I}$, then from Equation~\eqref{eq75} we have 
\begin{equation*}
  \nabla^{2} \phi_{n+1}(x)=(1-\pi_{n})\nabla^{2} \phi_{n}(x) + \pi_{n} \mu \mathbf{I} = \gamma_{n+1} \mathbf{I},
\end{equation*}
where we use Equation~\eqref{eq85}. Thus  $\nabla^{2} \phi_{n}(x)=\gamma_{n} \mathbf{I}$ for all $n$. Since $\phi_{n}(x)$ is a quadratic function, $\phi_{n}(x) = \phi_{n}\opt + \frac{\gamma_{n}}{2}\|x-v_{n}\|^{2}$, where $\phi_{n}\opt = \min_{x}\phi_{n}(x)$ and $v_{n}$ is the minimizer of $\phi_{n}(x)$. 
\paragraph{Step 2} Next, we give the recursive rule of $\phi_{n}\opt$ and $v_{n}$. Firstly, note that $\nabla \phi_{n}(x)=\gamma_{n} (x-v_{n})$. Then according to the first order optimality condition of $\min_{x}\phi_{n+1}(x)$ and Equation~\eqref{eq75}, we have
  \begin{align*}
    \nabla \phi_{n+1}(x) & = (1-\pi_{n})\nabla \phi_{n}(x) + \pi_{n}[\mu(x-y_{n})+\hat{G}_{\alpha}(y_{n})]\\
    & = \gamma_{n+1}x-(1-\pi_{n})\gamma_{n}v_{n}-\mu \pi_{n}y_{n}+\pi_{n}\hat{G}_{\alpha}(y_{n}) = 0,
  \end{align*}
which implies the recursive rule of $v_{n}$:
\begin{equation}\label{eq77}
  v_{n+1} = \frac{1}{\gamma_{n+1}}\left[(1-\pi_{n})\gamma_{n}v_{n}+\mu \pi_{n}y_{n}-\pi_{n}\hat{G}_{\alpha}(y_{n})\right].
\end{equation}
Now we compute $\phi_{n+1}\opt$:
\begin{align}
    \phi_{n+1}\opt &=  \phi_{n+1}(y_{n})-\frac{\gamma_{n+1}}{2}\norm{y_{n}-v_{n+1}}^{2} \nonumber \\
    &= (1-\pi_{n})\phi_{n}(y_{n})+\pi_{n} [F(x_{n+1})-\eta \norm{G_{\alpha}(y_{n})}^{2}]- \frac{\gamma_{n+1}}{2}\norm{y_{n}-v_{n+1}}^{2}. \label{eq78}
\end{align}
By Equation~\eqref{eq77} we have $ v_{n+1}-y_{n} = \frac{1}{\gamma_{n+1}}[(1-\pi_{n})\gamma_{n}(v_{n}-y_{n})-\pi_{n}\hat{G}_{\alpha}(y_{n})]$.
Thus,
  \begin{align}
 \frac{\gamma_{n+1}}{2}\bignorm{y_{n}-v_{n+1}}^{2}=&\frac{1}{2\gamma_{n+1}}\left[(1-\pi_{n})^{2}\gamma_{n}^{2}\bignorm{v_{n}-y_{n}}^{2}+\pi_{n}^{2}\bignorm{\hat{G}_{\alpha}(y_{n})}^{2}  -2\pi_{n}(1-\pi_{n})\gamma_{n}\iproduct{v_{n}-y_{n}}{\hat{G}_{\alpha}(y_{n})}\right] \label{eq86}
  \end{align}
Substituting this into Equation~\eqref{eq78} and recalling that $\phi_{n}(x) = \phi_{n}\opt + \frac{\gamma_{n}}{2}\bignorm{x-v_{n}}^{2}$, we have
  \begin{align}
   \nonumber  \phi_{n+1}\opt  = {}& (1-\pi_{n})\left[\phi_{n}\opt+\frac{\gamma_{n}}{2}\bignorm{y_{n}-v_{n}}^{2}\right]+ \pi_{n} \left[F(x_{n+1})-\eta \bignorm{G_{\alpha}(y_{n})}^{2}\right]\\
\nonumber     & -\frac{1}{2\gamma_{n+1}}\left[(1-\pi_{n})^{2}\gamma_{n}^{2}\bignorm{v_{n}-y_{n}}^{2}+\pi_{n}^{2}\bignorm{\hat{G}_{\alpha}(y_{n})}^{2}  -2\pi_{n}(1-\pi_{n})\gamma_{n}\iproduct{v_{n}-y_{n}}{\hat{G}_{\alpha}(y_{n})}\right]\\
     \nonumber ={}& (1-\pi_{n})\phi_{n}\opt+\pi_{n}F(x_{n+1})-\eta \pi_{n}\bignorm{G_{\alpha}(y_{n})}^{2}\\
      & + \frac{\pi_{n}(1-\pi_{n})\gamma_{n}}{\gamma_{n+1}}\left[\frac{\mu}{2}\bignorm{v_{n}-y_{n}}^{2}+\iproduct{v_{n}-y_{n}}{\hat{G}_{\alpha}(y_{n})}\right] -\frac{\pi_{n}^{2}}{2\gamma_{n+1}}\bignorm{\hat{G}_{\alpha}(y_{n})}^{2}, \label{eq79}
  \end{align}
  where we use Equation~\eqref{eq85} to compute the coefficient of the term $\|v_{n}-y_{n}\|^{2}$.
\paragraph{Step 3} Now we show that $\mbE_{n-1}[\phi_{n}\opt]\geq\mbE_{n-1}[F(x_{n})]$, which also implies $\mbE[\phi_{n}\opt]\geq\mbE[F(x_{n})]$. According to Equation~\eqref{eq74}, this holds for $n=0$. We prove the conclusion by induction. Suppose that $\mbE_{n-1}[\phi_{n}\opt]\geq\mbE_{n-1}[F(x_{n})]$, then by Equation~\eqref{eq79} we have
  \begin{align}
  \nonumber  \mbE_{n}[\phi_{n+1}\opt]  = {}&  (1-\pi_{n})\mbE_{n}[\phi_{n}\opt]+ \mbE_{n}\left[\pi_{n} F(x_{n+1}) - \pi_{n}\eta \bignorm{G_{\alpha}(y_{n})}^{2} -\frac{\pi_{n}^{2}}{2\gamma_{n+1}}\bignorm{\hat{G}_{\alpha}(y_{n})}^{2}\right.\\
   \nonumber & \left. \quad \qquad \qquad \qquad\qquad+ \frac{\pi_{n}(1-\pi_{n})\gamma_{n}}{\gamma_{n+1}}\left[\frac{\mu}{2}\norm{v_{n}-y_{n}}^{2}+\iproduct{v_{n}-y_{n}}{\hat{G}_{\alpha}(y_{n})}\right]\right]\\
  \nonumber  \overset{\text{(Induction)}}{\geq}{}& (1-\pi_{n})\mbE_{n}[F(x_{n})]+ \mbE_{n}\left[\pi_{n} F(x_{n+1}) - \pi_{n}\eta \bignorm{G_{\alpha}(y_{n})}^{2} -\frac{\pi_{n}^{2}}{2\gamma_{n+1}}\bignorm{\hat{G}_{\alpha}(y_{n})}^{2}\right.\\
  \nonumber  & \left. \,\qquad\qquad \qquad \qquad\qquad+ \frac{\pi_{n}(1-\pi_{n})\gamma_{n}}{\gamma_{n+1}}\left[\frac{\mu}{2}\norm{v_{n}-y_{n}}^{2}+\iproduct{v_{n}-y_{n}}{\hat{G}_{\alpha}(y_{n})}\right]\right]\\
  \nonumber  \overset{\text{(Lemma~\ref{lemma:adaNAPG sufficient descent})}}{\geq}{}& (1-\pi_{n})\mbE_{n} \left[ F(x_{n+1}) + \frac{\mu}{2}\bignorm{y_{n}-x_{n}}^{2}-\eta\norm{G_{\alpha}(y_{n})}^{2}+\iproduct{\hat{G}_{\alpha}(y_{n})}{x_{n}-y_{n}} \right]\\
  \nonumber  &+ \mbE_{n}\left[\pi_{n} F(x_{n+1}) - \pi_{n}\eta \bignorm{G_{\alpha}(y_{n})}^{2} -\frac{\pi_{n}^{2}}{2\gamma_{n+1}}\bignorm{\hat{G}_{\alpha}(y_{n})}^{2}\right.\\
   \nonumber & \left.  \quad\qquad+ \frac{\pi_{n}(1-\pi_{n})\gamma_{n}}{\gamma_{n+1}}\left[\frac{\mu}{2}\norm{v_{n}-y_{n}}^{2}+\iproduct{v_{n}-y_{n}}{\hat{G}_{\alpha}(y_{n})}\right]\right]\\
   \nonumber \geq {} & \mbE_{n}\left[F(x_{n+1})- \eta \bignorm{G_{\alpha}(y_{n})}^{2}-\frac{\pi_{n}^{2}}{2\gamma_{n+1}}\bignorm{\hat{G}_{\alpha}(y_{n})}^{2}\right]\\
  \nonumber  & + (1-\pi_{n})\mbE_{n}\left[\iproduct{\hat{G}_{\alpha}(y_{n})}{x_{n}-y_{n}+\frac{\pi_{n}\gamma_{n}}{\gamma_{n+1}}(v_{n}-y_{n})}\right]. \label{eq80}
  \end{align}
Following similar arguments in the proof of Lemma 2 in~\cite{lei2024variance} we can show that $x_{n}-y_{n}+\frac{\pi_{n}\gamma_{n}}{\gamma_{n+1}}(v_{n}-y_{n})=0$. Combined with Equation~\eqref{eq72},
  \begin{align*}
    \mbE_{n}[\phi_{n+1}\opt] \geq{} & \mbE_{n}\left[F(x_{n+1})- \eta \bignorm{G_{\alpha}(y_{n})}^{2}-\frac{\pi_{n}^{2}}{2\gamma_{n+1}}\bignorm{\hat{G}_{\alpha}(y_{n})}^{2}\right]\\
    \geq{} & \mbE_{n}\left[F(x_{n+1})- \eta \bignorm{G_{\alpha}(y_{n})}^{2}-\frac{\pi_{n}^{2}(\theta^{2}+\nu^{2}+1)}{\gamma_{n+1}}\bignorm{G_{\alpha}(y_{n})}^{2}\right].
  \end{align*}
If choose $\alpha=[(\theta^{2}+\nu^{2}+1)L]^{-1}$, we have $\eta = \alpha[(\alpha L -2)(\theta^{2}+\nu^{2}+1)+\theta^{2}+\nu^{2}]=-1/L$. From Equations~\eqref{eq82} and~\eqref{eq85}, it holds that $\eta + \frac{\pi_{n}^{2}(\theta^{2}+\nu^{2}+1)}{\gamma_{n+1}}=0$. Hence, it holds that $\mbE_{n}[\phi_{n+1}\opt] \geq\mbE_{n}[F(x_{n+1})]$.

\paragraph{Step 4} Now we show that for any $z$, $\mbE[\phi_{n}(z)] \leq (1-\lambda_{n})F(z) + \lambda_{n}\phi_{0}(z)$ holds, where $\lambda_{n}=\prod_{i=1}^{n}(1-\pi_{i})$. We prove this by induction. In fact by Lemma~\ref{lemma:adaNAPG sufficient descent},
  \begin{align*}
    \mbE[\phi_{n+1}(z)] & = (1-\pi_{n})\mbE[\phi_{n}(z)] + \pi_{n}\mbE\left[F(x_{n+1})-\eta \bignorm{G_{\alpha}(y_{n})}^{2} + \frac{\mu}{2}\bignorm{x-y_{n}}^{2}+\iproduct{\hat{G}_{\alpha}(y_{n})}{x-y_{n}}\right]\\
    & \leq (1-\pi_{n})\mbE[\phi_{n}(z)] + \pi_{n}F(z)\\
    & = [1-(1-\pi_{n})\lambda_{n}]F(z) + (1-\pi_{n})\mbE[\phi_{n}(z)-(1-\lambda_{n})F(z)]\\
    & \leq (1-\lambda_{n+1})F(z) + (1-\pi_{n})\lambda_{n}\phi_{0}(z) = (1-\lambda_{n+1})F(z) + \lambda_{n+1}\phi_{0}(z).
  \end{align*}

\paragraph{Step 5} From Step 3 and Step 4 we have
  \begin{align*}
    \mbE[F(x_{n})] \leq  \mbE[\phi_{n}\opt]  \leq & \mbE[\min_{z}\phi_{n}(z)]
    \leq  \min_{z} \mbE[\phi_{n}(z)]
     \leq  \min_{z} (1-\lambda_{n})F(z) + \lambda_{n}\phi_{0}(z)\\ 
     \leq & (1-\lambda_{n})F(x\opt) + \lambda_{n}\phi_{0}(x\opt). \Halmos
  \end{align*}
\end{proof}

\subsection{Proof of Theorem~\ref{adaNAPG_convergence}}
Lemma~\ref{lemma:estimate sequence} shows that to prove Theorem~\ref{adaNAPG_convergence}, it is sufficient to compute the convergence rate of $\lambda_{n}$. Firstly, suppose that  $f$ in Problem~\eqref{problem} is only convex ($\mu=0$), we choose $\gamma_{0}=L$ in Equation~\eqref{eq74}. Then we can prove that $\gamma_{n}\geq L \lambda_{n}$ for all $n$ by induction, since
  $\gamma_{n+1} = (1-\pi_{n})\gamma_{n}+\pi_{n}\mu\geq (1-\pi_{n})L\lambda_{n}= L\lambda_{n+1}$. By Equation~\eqref{eq82} it holds that $L(\theta^{2}+\nu^{2}+1)\pi_{n}^{2}=\gamma_{n+1}\geq L\lambda_{n+1}$,
which implies that
\begin{equation*}
  \frac{\pi_{n}}{\sqrt{\lambda_{n+1}}} \geq \sqrt{\frac{1}{L(\theta^{2}+\nu^{2}+1)}}.
\end{equation*}
Denote $a_{n}=1/{\sqrt{\lambda_{n}}}$, then according to the similar arguments in the proof of Lemma 2.2.4 in \cite{nesterov2013introductory} one can get $a_{n} \geq 1+\frac{n}{2\sqrt{(\theta^{2}+\nu^{2}+1)}}$, which implies that
\[
\lambda_{n} \leq \frac{4(\theta^{2}+\nu^{2}+1)}{(2\sqrt{(\theta^{2}+\nu^{2}+1)}+n)^{2}}.
\]

Now suppose that $f$ in Problem~\eqref{problem} is strongly convex with parameter $\mu>0$, we choose $\gamma_{0}=\mu$ in Equation~\eqref{eq74}. Then it is easy to prove $\gamma_{n}\geq \mu$ for all $n$ by induction. Besides, we also have $\pi_{n} \geq \sqrt{\frac{\mu}{L(\theta^{2}+\nu^{2}+1)}}$. Recalling that $\lambda_{n}=\Pi_{i=1}^{n}(1-\pi_{n})$, we have
\begin{equation*}
  \lambda_{n} \leq \left(1-\sqrt{\frac{\mu}{L(\theta^{2}+\nu^{2}+1)}}\right)^{n}.
\end{equation*}
Hence Theorem~\ref{adaNAPG_convergence} is proved. \Halmos

\subsection{Proof of Theorem~\ref{adaNAPG_simulation}}
For the strongly convex setting, Algorithm~\ref{algorithm_adaNAPG} gives us
\begin{equation}\label{eq233}
    y_{n+1}-x\opt=x_{n+1}+\beta(x_{n+1}-x_{n}) - x\opt= (1+\beta)(x_{n+1}-x\opt) + (-\beta)(x_{n}-x\opt),
\end{equation}
where $\beta=(1-\sqrt{\alpha\mu})/(1+\sqrt{\alpha\mu})$. 
From Theorem~\ref{adaNAPG_convergence} we know that $\mbE[\|{x_n - x\opt}\|^{2}]=O(\rho^n)$ when the objective is strongly convex, which together with Equation\eqref{eq233} implies that $\mbE[\|{y_n - x\opt}\|^{2}]=O(\rho^n)$. Since the gradient mapping $G_{\alpha}(\cdot)$ is $(2\alpha + L)$-Lipschitz continuous~\citep[Lemma 10.10,][]{Beck2017First},  we have 
\begin{equation}\label{eq2334}
    \mbE\left[\bignorm{G_{\alpha}(y_{n})}^{2}\right]  =  \mbE\left[\bignorm{G_{\alpha}(y_{n})-G_{\alpha}(x\opt)}^{2}\right] 
    \leq (2\alpha + L)^{2}\mbE\left[\bignorm{y_{n}-x\opt}^{2}\right]=O(\rho^n).
\end{equation}
 Lemma 6.5 in \cite{Pasupathy2018sampling} further gives that $\mbE[\norm{G_{\alpha}(y_{n})}^{2}]=O_{\mathbb{P}}(\rho^{-n})$. Note that the left hand of Equation~\eqref{eq104} equals to $\frac{1}{K_{n}} \Var[g(y_{n},\xi_{n}^{1})]$ since $g(y_{n},\xi_{n}^{1})$ is unbiased.
Thus, from Test~\eqref{NAPG_test1} and Test~\eqref{NAPG_test2} we have that $K_{n} = \inf\{K\geq 1: g_{n} / \sigma^2 \leq \norm{G_{\alpha}(y_n)}^{2} (\theta^2+\nu^2)K/\sigma^2  \}$. According to the Lemma 1 in \cite{Chow1965asymptotic}, we have that
\begin{equation*}
    \lim_{n \to \infty} \frac{\theta^2+\nu^2}{\sigma^2}K_{n} \bignorm{G_{\alpha}(y_n)}^{2} = 1 \qquad \text{a.s.}
\end{equation*}
Combined with$\mbE[\norm{G_{\alpha}(y_{n})}^{2}]=O_{\mathbb{P}}(\rho^{-n})$, we have $K_{n} = O_{\mathbb{P}}(\rho^{-n})$ and thus $\Gamma_{n} = O_{\mathbb{P}}(\rho^{-n})$. To finish the proof, one may only need to notice that $\mbE[\|{x_n - x\opt}\|^{2}]=O(\rho^n)$ and again use Lemma 6.5 in \cite{Pasupathy2018sampling}. \Halmos

\section{Proofs in Section~\ref{sec:asymptotic}}\label{app:asymptotic}

In this section, we prove Lemma~\ref{lemma_linear_recursion} and Theorem~\ref{CLT_adaNAPG} in Section~\ref{sec:asymptotic}. Before that, we prove some auxiliary lemmas.

\subsection{Auxiliary Lemmas}
Lemma~\ref{lemma_G_Jacobian} gives a linear approximation of the gradient mapping with controlled error. For a matrix $A$, $\|A\|_{2}$ and $\|A\|_{F}$ represent its 2-norm and Frobenius norm, respectively.

\begin{lemma}\label{lemma_G_Jacobian}
  Suppose that Assumptions~\ref{assumption_problem} and~\ref{assum_f_sc}--\ref{assum_G_semismooth} hold. Let $x\opt$ be the optimal solution of Problem~\eqref{problem} and $0<\alpha \leq \frac{1}{L}$. Then for any $x$ there exits a matrix $H_x$ depending on $x$ such that $G_{\alpha}(x) = H_x(x-x\opt) + \delta(x)$, where $\delta(x\opt)=0$ and $\delta(x)=o(\|{x - x\opt}\|)$.
 Moreover, we have $\|{\mathbf{I}-\alpha H_x}\|_{2}\leq 1-\alpha\mu<1$ for all $x$.
\end{lemma}
\begin{proof}{Proof of Lemma~\ref{lemma_G_Jacobian}}
Since $\alpha \leq \frac{1}{L}$ and $\mu\leq L$, we have
\begin{equation}\label{eq91}
    \norm{\mathbf{I}-\alpha\nabla^{2}f(x)}_{2}\leq \max\{||1-\alpha L||,|1-\alpha \mu|\}=1-\alpha \mu.
\end{equation}
  Since $G_{\alpha}(x)$ is Lipschitz continuous, the Clarke generalized Jacobian $\partial G_{\alpha}(x)$ at $x\opt$ exists. By assumption, $G_{\alpha}(x)$ is semismooth at $x\opt$, thus according to Lemma~\ref{lemma_semismooth1} we have $G_{\alpha}(x) - G_{\alpha}(x\opt) = G_{\alpha}^{\prime}(x\opt; x-x\opt) + o(\|{x_n - x\opt}\|)$. From Lemma~\ref{lemma_semismooth3}, there exists a $H_{x} \in \partial G_{\alpha}(x\opt)$ such that $H_{x}(x-x\opt)=G_{\alpha}^{\prime}(x\opt; x-x\opt)$. Thus 
    $ G_{\alpha}(x) = H_{x}(x-x\opt) + o(\norm{x-x\opt}),$
    where we use the fact that $G_{\alpha}(x\opt)=0$. Next we calculate the bound of $\norm{\mathbf{I}-\alpha H_{x}}_2$. To simplify the notation, define an auxiliary function $ s(x):=\prox_{\alpha h}(x)$. Note that $s(x)$ is Lipschitz continuous. Moreover, according to Proposition 2.6.2 in \cite{clarke1990optimization} for any $J \in \partial s(x)$ one has $\norm{J}_{F}\leq 1$. 
Recall that the gradient mapping $G_{\alpha}(x)$ is defined as $G_{\alpha}(x)=\frac{1}{\alpha}[x-\prox_{\alpha h}(x-\alpha \nabla f(x))]$. By Theorem 2.6.6 in \cite{clarke1990optimization} we know that the Generalized Jacobian of $G_{\alpha}(x)$ at $x\opt$ satisfies 
    \begin{equation*}
      \partial G_{\alpha}(x\opt) \subseteq \co \left\{ \frac{1}{\alpha} \left[\mathbf{I} - J(\mathbf{I}-\alpha\nabla^{2}f(x\opt))\right] \mid J \in \partial s(x\opt - \alpha\nabla f(x\opt)) \right\},
    \end{equation*}
where $\co S$ denotes the convex hull of set $S$. Hence for any $H \in \partial G_{\alpha}(x\opt)$, it takes the form of $ H = \sum_{i}\lambda_{i}V_{i}$, where $\lambda_{i} \geq 0$, $\sum_{i}\lambda_{i}=1$ and $V_{i}=\frac{1}{\alpha} [\mathbf{I} - J_{i}(\mathbf{I}-\alpha\nabla^{2}f(x\opt))]$ for some $J_{i} \in \partial s(x\opt - \alpha\nabla f(x\opt))$. Now one may have
\[ 
        \bignorm{\mathbf{I}-\alpha H}_{2} = \left\|{\mathbf{I}-\alpha \sum_{i}\lambda_{i}V_{i}}\right\|_{2} =  \bignorm{\sum_{i}\lambda_{i}J_{i}(\mathbf{I}-\alpha\nabla^{2}f(x\opt))}_{2} 
        \leq \bignorm{\sum_{i}\lambda_{i}J_{i}}_{2} \bignorm{\mathbf{I}-\alpha\nabla^{2}f(x\opt)}_{2}.
\]
Since for any matrix $M$ it holds that $\norm{M}_{2}\leq \norm{M}_{F}$ (this is because 2-norm is the largest singular value of $M$ while F-norm is the square root of the sum of all squared singular values of $M$), we have $\norm{J_{i}}_{2} \leq 1$, and thus $\norm{\sum_{i}\lambda_{i}J_{i}}_{2}\leq 1$. And  Equation~\eqref{eq91} gives $\norm{\mathbf{I}-\alpha H}_{2} \leq \norm{\mathbf{I}-\alpha\nabla^{2}f(x\opt)}_{2} \leq 1-\alpha \mu$. Thus, Lemma~\ref{lemma_G_Jacobian} is proved. \Halmos
\end{proof}

Next Lemma~\ref{PG_w_lemma} shows that under Assumption~\ref{assum_clt_noise}, the sequence $\{w_{n}\}$ has similar properties to $\{\tilde{w}_{n}\}$. 
\begin{lemma}\label{PG_w_lemma}
    Suppose that Assumptions~\ref{assumption_gradient} and~\ref{assum_clt_noise} hold. Then the error sequence $w_{n}$ satisfies the following:
    \begin{enumerate}
    \item[(i)] There exists a semi-definite matrix $S_{0}$ such that 
    \begin{equation}\label{eq53}
        \lim_{n \to \infty}  \rho^{-n} \mbE[w_{n}w_{n}^{\top}] = S_{0}.
    \end{equation}
    \item[(ii)] The following Lindeberg's condition holds.
    \begin{equation}\label{eq54}
        \lim_{r \to \infty}\sup_{n}\mbE \left[\norm{\rho^{-n/2}w_{n}}^{2}\mathbb{I}_{[\norm{\rho^{-n/2}w_{n}}>r]} \right] = 0.
    \end{equation}
\end{enumerate}
\end{lemma} 
\begin{proof}{Proof of Lemma~\ref{PG_w_lemma}}
    By the definition of $w_{n}$,
\begin{align}
    \nonumber        w_{n} & = \hat{G}_{\alpha}(y_{n})-G_{\alpha}(y_{n}) \\
    \nonumber    & = \frac{1}{\alpha}\left[y_{n}-\prox_{\alpha h}(y_{n}-\alpha(\nabla f(y_{n})+\tilde{w}_{n}))\right] - \frac{1}{\alpha}\left[y_{n}-\prox_{\alpha h}(y_{n}-\alpha\nabla f(y_{n}))\right] \\
     & = \frac{1}{\alpha}\left[\prox_{\alpha h}(y_{n}-\alpha\nabla f(y_{n})) - \prox_{\alpha h}(y_{n}-\alpha(\nabla f(y_{n})+\tilde{w}_{n}))\right],\label{eq55}
         \end{align}
    which shows that $w_{n}$ is Lipschitz continuous on $\tilde{w}_{n}$ since the proximal operator itself is Lipschitz continuous. To simplify the notation, we define the function between $w_{n}$ and $\tilde{w}_{n}$ as $W$, that is, $w_{n} = W(\tilde{w}_{n})$. Noting that $W(0)=0$, the Taylor expansion of $W$ at point $0$ is given by $W(x)=W(0) + \nabla W(0)^{\top}x + o(\norm{x}) = \nabla W(0)^{\top}x + o(\norm{x})$. Thus, together with $\mbE[\tilde{w}_{n}]=0$, we have 
    \begin{align}
   \nonumber   \rho^{-n} \mbE[w_{n}w_{n}^{\top}] &= \rho^{-n} \mbE[W(\tilde{w}_{n})W(\tilde{w}_{n})^{\top}] \\
    \nonumber & =\rho^{-n} \mbE[(\nabla W(0)^{\top}\tilde{w}_{n} + o(\norm{\tilde{w}_{n}}))(\nabla W(0)^{\top}\tilde{w}_{n} + o(\norm{\tilde{w}_{n}}))^{\top}]  \\
    & =\rho^{-n} \mbE[\nabla W(0)^{\top}\tilde{w}_{n}\tilde{w}_{n}^{\top}W(0) +o(\norm{\tilde{w}_{n}}^{2})] \label{eq:ww_top}
    \end{align}
Adding tests \eqref{NAPG_test1} and \eqref{NAPG_test2} together and taking Equation~\eqref{eq2334} in to account gives
      \begin{equation*}
      \mbE\left[\bignorm{\tilde{w}_{n}}^{2}\right] \leq (\theta^{2} + \nu^{2})\mbE\left[\bignorm{G_{\alpha}(y_{n})}^{2}\right] = O(\rho^{n}),
      \end{equation*}
which implies that  $\mbE[\norm{\tilde{w}_{n}}^{2}]$ is in order of $O(\rho^{n})$. 
Together with Equation~\eqref{eq51} and \eqref{eq:ww_top}, we obtain
\begin{align*}
    \lim_{n \to \infty} \rho^{-n} \mbE[w_{n}w_{n}^{\top}] & = \lim_{n \to \infty} \rho^{-n} \mbE[\nabla W(0)^{\top}\tilde{w}_{n}\tilde{w}_{n}^{\top}W(0) +o(\norm{\tilde{w}_{n}}^{2})]  \\ 
    &  = \lim_{n \to \infty} \rho^{-n} W(0)^{\top} \mbE[\tilde{w}_{n}\tilde{w}_{n}^{\top}]W(0)  = W(0)^{\top} \tilde{S}_{0} W(0),
\end{align*}
which proves Equation~\eqref{eq53}. Again with Equation~\eqref{eq55} and the Lipschitz continuity of the proximal operator, we have $\norm{w_{n}}  \leq \norm{\tilde{w}_{n}}$. Thus for any $r>0$, we have $ \{ \norm{w_{n}} > r \} \subset \{ \norm{\tilde{w}_{n}} > r \}$, 
which implies that $\mathbb{I}_{\{ \norm{w_{n}} > r \}} \leq \mathbb{I}_{\{ \norm{\tilde{w}_{n}} > r \}}$. Hence we have
\[\mbE \left[\norm{\rho^{-n/2}w_{n}}^{2}\mathbb{I}_{[\norm{\rho^{-n/2}w_{n}}>r]} \right]\leq \mbE \left[\norm{\rho^{-n/2}\tilde{w}_{n}}^{2}\mathbb{I}_{[\norm{\rho^{-n/2}\tilde{w}_{n}}>r]} \right],
\]
which, together with Equation~\eqref{eq55}, implies that \eqref{eq54} holds. \Halmos
\end{proof}

The following Lemma~\ref{lemma_linear_recursion} shows that the error sequence $\{\zeta_n\}$ in Lemma~\ref{lemma_linear_recursion} is asymptotically negligible. The notation $o_{\mathbb{P}}(1)$ is short for a sequence that converges to zero in probability.

\begin{lemma}\label{lemma:zeta_n}
    The sequence $\{\zeta_n\}$ defined in Lemma~\ref{lemma_linear_recursion} is $o_{\mathbb{P}}(1)$.
\end{lemma}
\begin{proof}{Proof of Lemma~\ref{lemma:zeta_n}}
     Define $D_{n} = \frac{\delta(y_n)(y_{n}-x\opt)^{\top}}{\norm{y_{n}-x\opt}^{2}}$ if  $y_{n} \neq x\opt$ and $D_n=0$ otherwise and define $\bar{e}_n=\rho^{-n/2}(y_n-x\opt)$.  Here $\delta(\cdot)$ is defined in Lemma~\ref{lemma_G_Jacobian}. Then we have
\begin{equation}\label{eq64}
  \delta(y_{n})=D_{n}(y_{n}-x\opt),
\end{equation}
which implies that $ \zeta_{n}=-\alpha\rho^{-(n+1)/2}D_{n}(y_{n}-x\opt)=-\alpha \rho^{-1/2} D_{n}\bar{e}_{n}$. In the following, we will separately prove that $\bar{e}_n=O_{\mathbb{P}}(1)$ and $D_{n}=o_{\mathbb{P}}(1)$. Then $\zeta_{n}=O_{\mathbb{P}}(1)o_{\mathbb{P}}(1)=o_{\mathbb{P}}(1)$~\citep[Section 2.2,][]{VanderVaart2000Asymptotic}.

Firstly, we show that $D_{n}=o_{\mathbb{P}}(1)$. Theorem~\ref{adaNAPG_convergence} together with Equation~\eqref{eq233} shows $\mbE[\norm{y_{n}-x\opt}^{2}]\to 0$, and thus $y_{n}-x\opt$ converges to $0$ in probability because of the Markov's inequality. Then by using $\delta(x\opt)=0$, $\delta(y)=o(\norm{y-x\opt})$ from Lemma~\ref{lemma_G_Jacobian}, we have $\delta(y_{n})=o_{\mathbb{P}}(\norm{y_{n}-x\opt})$. Together with Equation~\eqref{eq64}, $D_{n}(y_{n}-x\opt)=o_{\mathbb{P}}(\norm{y_{n}-x\opt})$. 
Recall that the definition of $o_{\mathbb{P}}(\cdot)$ is ``$ X_{n}=o_{\mathbb{P}}(R_{n}) \Leftrightarrow X_{n}=Y_{n}R_{n} \text{ and } Y_{n}$ converges to $0$ in probability "~\citep[Section 2.2,][]{VanderVaart2000Asymptotic}. Thus, $D_{n}=o_{\mathbb{P}}(1)$.

Then, we prove that $\bar{e}_n=O_{\mathbb{P}}(1)$. For any $n\geq 0$, Theorem~\ref{adaNAPG_convergence} shows that
\[
\Var[\bar{e}_n]\leq \mbE[\norm{\bar{e}_n}^{2}] = \mbE[\norm{y_{n}-x\opt}^{2}/\rho^{n}] \leq v_{e}^{2},
\]
for some constant $v_{e}$.
From Chebyshev’s inequality we know that if $X$ is a random variable with mean $\mu$ and variance $\sigma^{2}$, then for any real number $u>0$, $\Pr{\norm{X-\mu}\leq u\sigma} \geq 1-u^{-2}$. Setting $u=\chi^{-1/2}$ for any $\chi \in (0,1)$ then we have $\Pr{\norm{\bar{e}_n-\mbE[\bar{e}_n]}\leq \chi^{-1/2}\sqrt{\Var[\bar{e}_n]}} \geq 1-\chi$. By Jensen's inequality, we have $ \norm{\mbE[\bar{e}_n]} \leq \mbE[\norm{\bar{e}_n}] \leq \sqrt{\mbE[\norm{\bar{e}_n}^{2}]} \leq v_{e}$. Thus, we have $\norm{\bar{e}_n-\mbE[\bar{e}_n]} \geq \norm{\bar{e}_n}-\norm{\mbE[\bar{e}_n]}  \geq \norm{\bar{e}_n}-v_{e}$, which implies that $\left\{ \norm{\bar{e}_n-\mbE[\bar{e}_n]} \leq \chi^{-1/2}v_{e}  \right\} \subset \left\{ \norm{\bar{e}_n} \leq v_{e} + \chi^{-1/2}v_{e}\right\}$. Moreover, since $\sqrt{\Var[\bar{e}_n]} \leq v_{e}$, it follows that 
\[
\left\{ \norm{\bar{e}_n-\mbE[\bar{e}_n]} \leq  \chi^{-1/2}\sqrt{\Var[\bar{e}_n]}  \right\} \subset \left\{ \norm{\bar{e}_n-\mbE[\bar{e}_n]} \leq  \chi^{-1/2}v_{e}  \right\} \subset \left\{ \norm{\bar{e}_n} \leq v_{e} + \chi^{-1/2}v_{e}\right\},
\]
which implies
$ \Pr{\norm{\bar{e}_n} \leq v_{e} + \chi^{-1/2}v_{e}} \geq \Pr{\norm{\bar{e}_n-\mbE[\bar{e}_n]} \leq  \chi^{-1/2}\sqrt{\Var[\bar{e}_n]} } \geq 1-\chi$. So $\bar{e}_n$ is bounded in probability. \Halmos
\end{proof}

Lemma~\ref{lemma_sum_converge} gives a basic result of sequence limit.
\begin{lemma}\label{lemma_sum_converge}
  Let $0<a<1$ be a real number and $\{b_{t}\}$ be a sequence such that $b_{t} \to 0$ as $t \to \infty$. Then we have $\sum_{t=1}^{k}a^{k-t}b_{t} \to 0 $ as $k \to \infty$. 
  \end{lemma}
  \begin{proof}{Proof of Lemma~\ref{lemma_sum_converge}} 
  Since $b_{t}\to0$ as $t \to \infty$, for any $\varepsilon>0$ there exits a $T>0$ such that when $t>T$, $|b_{t}|\leq \varepsilon$. Also, $b_{t}$ is bounded by some $M>0$, that is, $|b_{t}|\leq M$ for all $t$. Then for $k$ large enough,
    \begin{align*}
      \left|\sum_{t=1}^{k}a^{k-t}b_{t}\right| & = \left|\sum_{t=1}^{T}a^{k-t}b_{t} + \sum_{t=T+1}^{k}a^{k-t}b_{t}\right|
       \leq \sum_{t=1}^{T}a^{k-t}|b_{t}| + \sum_{t=T+1}^{k}a^{k-t}|b_{t}| \\
      & \leq M\sum_{t=1}^{T}a^{k-t} + \varepsilon \sum_{t=T+1}^{k}a^{k-t}  = M\sum_{t=k-T}^{k-1}a^{t} + \varepsilon \sum_{t=0}^{k-T-1}a^{t}
    \end{align*}
  Since $0<a<1$, we can verify that $\sum_{t=0}^{k-T-1}a^{t} < \sum_{t=0}^{\infty}a^{t}=\frac{1}{1-a}$ and $ \sum_{t=k-T}^{k-1}a^{t} \leq \varepsilon $ for some $k$ large enough. Thus, for sufficiently large $k$ we have $\left|\sum_{t=1}^{k}a^{k-t}b_{t}\right| \leq (M+\frac{1}{1-a})\varepsilon$. \Halmos
\end{proof}

\subsection{Proof of Lemma~\ref{lemma_linear_recursion}}
With notation $G_{\alpha}(y_{n})$ and $w_{n}$, the iteration in Algorithm~\ref{algorithm_adaNAPG} can be rewritten as $x_{n+1}=y_{n}-\alpha(G_{\alpha}(y_{n})+w_{n})$. It follows that $x_{n+1}-x\opt = (\mathbf{I}-\alpha H_{y_n})(y_{n}-x\opt)-\alpha(\delta(y_n)+w_{n})$, where $H_{y_n}$ and $\delta(y_n)$ are defined in Lemma~\ref{lemma_G_Jacobian} and we use $H_n$ and $\delta_{n}$ in the following for simplicity. Also $y_{n}-x\opt=x_{n}+\beta(x_{n}-x_{n-1})-x\opt=(1+\beta)(x_{n}-x\opt)-\beta(x_{n-1}-x\opt)$, which implies that
  \begin{align*}
    x_{n+1}-x\opt & = (\mathbf{I}-\alpha H_{n})[(1+\beta)(x_{n}-x\opt)-\beta(x_{n-1}-x\opt)]-\alpha(\delta_{n}+w_{n})\\
    & = (1+\beta)(\mathbf{I}-\alpha H_{n})(x_{n}-x\opt) - \beta(\mathbf{I}-\alpha H_{n})(x_{n-1}-x\opt)-\alpha(\delta_{n}+w_{n}).
  \end{align*}
Based on the above two equations, we obtain
\begin{equation*}
  \begin{pmatrix}
    x_{n+1}- x\opt \\ x_{n} - x\opt
  \end{pmatrix} = 
  \begin{pmatrix}
    (1+\beta)(\mathbf{I}-\alpha H_{n}) & \beta(\alpha H_{n}-\mathbf{I})\\\mathbf{I} & 0
\end{pmatrix}
  \begin{pmatrix}
    x_{n}- x\opt \\ x_{n-1} - x\opt
  \end{pmatrix} - \alpha
  \begin{pmatrix}
    \delta_{n}+w_{n} \\ 0
  \end{pmatrix}.
\end{equation*}
Noting the definition of  $e_{n}$ and $P_{n}$, we have
\begin{equation*}
    e_{n+1}  = P_{n}e_{n} - \alpha\rho^{-(n+1)/2}\begin{pmatrix}
      \delta_{n}+w_{n} \\ 0
    \end{pmatrix} \\
   = P_{n}e_{n} - \alpha\rho^{-(n+1)/2}\begin{pmatrix}
      w_{n} \\ 0
    \end{pmatrix} + \begin{pmatrix} \zeta_{n}\\0 \end{pmatrix}.\Halmos
\end{equation*}

\subsection{Proof of Theorem~\ref{CLT_adaNAPG}}

The proof is inspired by \cite{lei2024variance}. Define the sequence $\{u_{n}\}$ as  $u_{n+1} = P_{n}u_{n} - \alpha \rho^{-(n+1)/2}w_{n}$ with $u_{0}=0$. By combining the definition of $\{e_{n}\}$ and $\{u_{n}\}$, we obtain the recursion $ e_{n+1}-u_{n+1}=P_{n}\left(e_{n}-u_{n}\right)+\zeta_{n}$. Define  $\Psi_{n,t}:=P_{n}P_{n-1}\cdots P_{t}$ and $ \Psi_{t-1,t}:=\mathbf{I}$. According to Equation (3.1.8) of \cite{chen2005stochastic}, Assumption~\ref{assum_spectral radius_Pn} implies that there exits some constant $c$ and $0<\hat{\kappa}<1$ such that $\norm{\Psi_{n,t}}_{2}\leq c \hat{\kappa}^{n-t+1}$. With the notation $\Psi_{n,t}$ we have $e_{n+1}-u_{n+1}=\Psi_{n,0}(e_{0}-u_{0}) + \sum_{t=0}^{n}\Psi_{n,t+1}\zeta_{t}$.
Together with Lemma~\ref{lemma:zeta_n}, it holds that
\begin{equation*}\label{eq92}
    \bignorm{e_{n+1}-u_{n+1}} \leq \bignorm{\Psi_{n,0}}_{2}\bignorm{e_{0}-u_{0}} + \sum_{t=0}^{n}\bignorm{\Psi_{n,t+1}}_{2} \bignorm{\zeta_{t}} \leq c \hat{\kappa}^{n+1}\bignorm{e_{0}-u_{0}} + \sum_{t=0}^{n} c \hat{\kappa}^{n-t+1} o_{P}(1), 
\end{equation*}
which converges to $0$ in probability as $n$ tends to $+\infty$. According to the fact that ``$X_{n} \to X$ in distribution together with $\norm{X_{n}-Y_{n}}\to0$ in probability means $Y_{n}\to X$ in distribution "~\citep{VanderVaart2000Asymptotic}, $\{e_{n}\}$ and $\{u_{n}\}$ have the same limit distribution if exists. Hence, it is sufficient to consider the sequence$\{u_{n}\}$.

From the definition of $\{u_{n}\}$, we have
\begin{equation*}
        u_{n+1}  =  P_{n}u_{n} - \alpha \rho^{-(n+1)/2}w_{n}\ = \Psi_{n,0}u_{0} - \alpha \sum_{t=0}^{n}\Psi_{n,t+1}\rho^{-(t+1)}w_{t} =- \alpha \sum_{t=0}^{n}\Psi_{n,t+1}\rho^{-(t+1)}w_{t},
\end{equation*}
which implies $\alpha^{-1}u_{n} = -\sum_{t=0}^{n-1}\Psi_{n-1,t+1}\rho^{-(t+1)}w_{t} = -\sum_{t=0}^{n}\Psi_{n-1,t}\rho^{-t}w_{t-1}$. Define $\xi_{n,t}:=-\Psi_{n-1,t} \rho^{-t/2}  w_{t-1}$ for any $1\leq t \leq n$. Now we begin to check the conditions \eqref{eq56}--\eqref{eq60} and apply Lemma~\ref{double-indexed CLT}. By the definition of $w_n$, we have
\begin{equation*}
    w_n  = \hG_{\alpha}(y_n) - G_{\alpha}(y_n)  = \frac{1}{\alpha}[\prox_{\alpha h}(y_n-\alpha \nabla f(y_n)) - \prox_{\alpha h}(y_n-\alpha \nabla f(y_n) - \alpha \tilde{w}_n)]. 
\end{equation*}
Since the proximal operator is Lipschitz continuous, according to Lemma~\ref{lemma_mean_value} there exist some matrix $V_n$ such that
    \begin{align*}
    w_n & = \frac{1}{\alpha}[\prox_{\alpha h}(y_n-\alpha \nabla f(y_n)) - \prox_{\alpha h}(y_n-\alpha \nabla f(y_n) - \alpha \tilde{w}_n)]  \\
     & = \frac{1}{\alpha} V_{n}[(y_n-\alpha \nabla f(y_n)) - (y_n-\alpha \nabla f(y_n) - \alpha \tilde{w}_n)]  = V_n \tilde{w}_n.
    \end{align*}
Since $\mbE_{t}[\tilde{w}_t]=0$ (it comes from the unbiasedness of the gradient estimator), it holds that $\mbE_{t}[w_{t}]=0$. Thus, we have $\mbE[\xi_{n,t} \mid \xi_{n1}, \cdots, \xi_{n,t-1}]=0$, which means Equation~\eqref{eq56} holds. By using $\norm{\Psi_{n-1,t}}\leq c \hkappa^{n-t}$, it holds that
    \begin{align*}
        \mbE\left[\norm{\xi_{n,t}}^{2}\mid\xi_{n1}, \cdots, \xi_{n,t-1}\right] 
       & \leq \mbE\left[\norm{\Psi_{n-1,t}}_{2}^{2}\rho^{-t}\norm{w_{t-1}}^{2}\mid\xi_{n1}, \cdots, \xi_{n,t-1}\right] \\
       & \leq c^{2}\rho^{-t}\hkappa^{2(n-t)} \mbE_{t}\left[\norm{w_{t-1}}^{2}] \leq c^{2}(\theta^{2}+\nu^{2})\rho^{-t}\hkappa^{2(n-t)} \mbE_{t}[\norm{G_{\alpha}(y_{t-1})}^{2}\right],
    \end{align*}
Equation~\eqref{eq2334} further gives that ,
    \begin{align*}
    \mbE[\norm{\xi_{n,t}}^{2}] 
        & \leq c^{2}(\theta^{2}+\nu^{2})(2\alpha+L)^{2} \hkappa^{2(n-t)}\rho^{-t} * O(\rho^{t-1})  \leq  \hkappa^{2(n-t)} O(1).
    \end{align*}
Hence, it holds that
\begin{equation*}
    \sum_{t=1}^{n}\mbE[\norm{\xi_{n,t}}^{2}] \leq \sum_{t=1}^{n} \hkappa^{2(n-t)} O(1) \leq  \frac{1}{1-\hkappa^{2}} O(1),
\end{equation*}
which implies that $\sup_{n} \sum_{t=1}^{n}\mbE[\norm{\xi_{n,t}}^{2}] < \infty$, that is, Equation~\eqref{eq57} holds.
Since $\xi_{n,t}=-\Psi_{n-1,t} \rho^{-t/2}  w_{t-1}$ and $\{w_{t}\}$ are independent, we have $  R_{nt}=\mbE[\xi_{n,t}\xi_{n,t}^{\top}\mid \xi_{n1},\cdots,\xi_{n,t-1}]=\mbE[\xi_{n,t}\xi_{n,t}^{\top}]=S_{nt}$, which verifies the condition \eqref{eq59}. We now verify the condition \eqref{eq58},
    \begin{align}
     \nonumber   S_{n}&=\sum_{t=1}^{n}S_{nt} =\sum_{t=1}^{n}\mbE[\xi_{n,t}\xi_{n,t}^{\top}]
         = \sum_{t=1}^{n}\mbE[\Psi_{n-1,t} \rho^{-t} w_{t-1}w_{t-1}^{\top}]\Psi_{n-1,t}^{\top}\\
        & = \sum_{t=1}^{n}\mbE[\Psi_{n-1,t} S_{0} \Psi_{n-1,t}^{\top}] + \sum_{t=1}^{n} \mbE[\Psi_{n-1,t} (\rho^{-t} w_{t-1}w_{t-1}^{\top}-S_{0})\Psi_{n-1,t}^{\top}].\label{eq61}
    \end{align}

For the second term in the right hand side of Equation~\eqref{eq61}, we have 
  \begin{align*}
  \left \| \sum_{t=1}^{n} \mbE \left[\Psi_{n-1,t} (\rho^{-t} w_{t-1}w_{t-1}^{\top}-S_{0})\Psi_{n-1,t}^{\top} \right] \right \|_{2}  & \leq \sum_{t=1}^{n}\mbE\left[ \left\|{\Psi_{n-1,t}}\right\|_{2}^{2}  \left\|{\rho^{-t} w_{t-1}w_{t-1}^{\top}-S_{0}}\right\|_{2}\right]\\
  & \leq  \sum_{t=1}^{n}c^{2}\hkappa^{2(n-t+1)}\mbE \left[\left\|{\rho^{-t} w_{t-1}w_{t-1}^{\top}-S_{0}}\right\|_{2}\right] 
  \end{align*}
By Equation~\eqref{eq53} and Lemma~\ref{lemma_sum_converge}, the second term on the right hand side of Equation~\eqref{eq61} converges to $0$. For the first term in the right hand side of Equation~\eqref{eq61}, we have
\begin{equation*}
\sum_{t=1}^{n}\norm{\Psi_{n-1,t} S_{0} \Psi_{n-1,t}^{\top}}_{2}\leq \norm{S_{0}}_{2} \sum_{t=0}^{n}\norm{\Psi_{n-1,t}}_{2}^{2} \leq c^{2}\norm{S_{0}}_{2} \sum_{t=0}^{n}\hkappa^{2n}\leq c^{2}\norm{S_{0}}_{2} / (1-\hkappa^{2}).
\end{equation*}
Since $\sum_{t=1}^{n}\norm{\Psi_{n-1,t}  S_{0} \Psi_{n-1,t}^{\top}}_{2}$ is monotonically increasing and bounded, its limit exists.  Thus, from Equation~\eqref{eq61} we can conclude that the limit of $S_{n}$ exists and condition~\eqref{eq58} holds. Besides we denote this limit as $\Sigma$. Finally, we verify the condition \eqref{eq60}. By $\xi_{n,t}=-\Psi_{n-1,t} \rho^{-t/2} w_{t-1}$, we have $\norm{\xi_{n,t}} \leq \norm{\Psi_{n-1,t}}_{2}\rho^{-t/2} \norm{w_{t-1}} \leq c\hkappa^{n-t} \rho^{-t/2}  \norm{w_{t-1}}$. Thus we have 
\begin{equation}\label{eq62}
\sum_{t=1}^{n}\mbE\left[\norm{\xi_{n,t}}^{2}\mathbb{I}_{[\norm{\xi_{n,t}}>\varepsilon]} \right]
    \leq \sum_{t=1}^{n}\kappa^{2(n-t)}\mbE \left[\norm{\rho^{-t/2}w_{t-1}}^{2}\mathbb{I}_{[\norm{\rho^{-t/2}w_{t-1}}>c^{-1} \hkappa^{t-n}\varepsilon]}\right]
\end{equation}
since for any $\varepsilon>0$, 
\[
 \{\norm{\xi_{n,t}}>\varepsilon\}  \subset \{c\hkappa^{n-t} \rho^{-t/2}  \norm{w_{t-1}} >\varepsilon\} = \{ \rho^{-t/2}  \norm{w_{t-1}} > c^{-1}\hkappa^{t-n}\varepsilon\}.
\]
By Equation~\eqref{eq54} we know that  $\sup_{n}\mbE \left[\norm{\rho^{-n/2}w_{n}}^{2}\mathbb{I}_{[\norm{\rho^{-n/2}w_{n}}>r]} \right] \to 0$ as $r \to 0$. Because for any $t\geq 1$, $\kappa^{t-n}\to \infty$ as $n \to \infty$. Thus,
\[
\sup_{t}\mbE \left[\norm{\rho^{-t/2}w_{t-1}}^{2}\mathbb{I}_{[\norm{\rho^{-t/2}w_{t-1}}>c^{-1}\hkappa^{t-n}\varepsilon]}\right] \to 0 \quad \text{ as } n \to 0.
\]
Together with Lemma~\ref{lemma_sum_converge} and Equation~\eqref{eq62}, it follows that $  \lim_{n \to \infty} \sum_{t=1}^{n}\mbE[\norm{\xi_{n,t}}^{2}\mathbb{I}_{[\norm{\xi_{n,t}}>\varepsilon]}] = 0$. We have verified all the conditions of Lemma~\ref{double-indexed CLT}, and thus $\alpha^{-1}u_{n}$ converges to  $\mathrm{Normal}(0,\Sigma)$ in distribution for some covariance matrix $\Sigma$. Since $e_{n}$ and $u_{n}$ have the same limit distribution, Theorem~\ref{CLT_adaNAPG} is proved.\Halmos

\end{APPENDICES}

\bibliographystyle{informs2014} 
\bibliography{sample} 
\ECSwitch

\ECHead{E-companion  for ``Boosting Accelerated Proximal Gradient Method with Adaptive Sampling for Stochastic Composite Optimization"}
\vspace{5mm}

\section{Auxiliary Nonsmooth Analysis Results}\label{EC:Auxiliary Nonsmooth Analysis Results}

In this section, we review some basic results of nonsmooth analysis.
These results are useful in the proof presented in Appendix~\ref{app:asymptotic}. Let $\partial G(x)$ be the generalized Jacobian defined in 2.6 of \citeec{clarke1990optimization}. Now we give the definition of \textit{semismooth} functions.

\begin{definition}
    We say that a function $G$ is semismooth at $x$ if $G$ is locally Lipschitz at $x$ and
    \begin{equation*}
        \lim_{{\begin{array}{c}V\in\partial G(x+ty^{\prime})\\y^{\prime}\to y,t\downarrow0\end{array}}}  \{V{y}^{\prime}\}
    \end{equation*}
    exists for any direction $y \in \mbR^{d}$.
\end{definition}

The following several lemmas summarize some basic properties of semismooth functions. See \citeec{Facchinei2003Finite} for more details.

\begin{lemma}\label{lemma_semismooth1}
    If $G$ is semismooth at $x$, then for any $y \to 0$ we have
    \begin{equation*}
        G(x+y)-G(x)=G^{\prime}(x;y)+o(\norm{y}),
    \end{equation*}
    where $G^{\prime}(x;y)$ is the classic directional derivative defined as
    \begin{equation*}
        G^{\prime}(x;y)=\lim_{t\downarrow 0}\frac{G(x+ty)-G(x)}{t}.
    \end{equation*}
\end{lemma}

\begin{lemma}\label{lemma_semismooth3}
    Suppose that $G$ is a locally Lipschitz function and $G^{\prime}(x;y)$ exists for any direction $y$ at $x$. Then for any $y$ there exists a $V \in \partial G(x)$ such that
    \begin{equation*}
        Vy=G^{\prime}(x;y).
    \end{equation*}
\end{lemma}

Next lemma is a generalization of the mean value theorem for nonsmooth functions. 
\begin{lemma}\label{lemma_mean_value}
     Let a function $G:\Omega \subset \mbR^{n}\to \mbR^{m}$ be Lipschitz continuous on an open set $\Omega$ containing the segment $[x, \bar{x}]$. There exist $m$ points $z^i$ in $(x, \bar{x})$ and $m$ scalars $\alpha_i \geq 0$ summing to unity such that
    \begin{equation*}
        G(\bar{x})=G(x)+\sum_{i=1}^{m}\alpha_{i}H_{i}(\bar{x}-x),
    \end{equation*}
    where, for each $i$, $H_{i} \in \partial G(z^{i})$.
\end{lemma}

\section{Double-indexed CLT}\label{EC:DCLT}
We present the double-indexed central limit theorem~\citep{chen2005stochastic} used in the proof of Theorem~\ref{CLT_adaNAPG}.
\begin{lemma}\label{double-indexed CLT}
  Let $\{\xi_{ki},1\leq i \leq k, k=1,2,\dots \}$ be an array of random vectors. Denote
$S_{ki}:=\mbE[\xi_{ki}\xi_{ki}^{\top}]$, $S_{k}:=\sum_{i=1}^{k}S_{ki}$, and
$R_{ki}:=\mbE[\xi_{ki}\xi_{ki}^{\top} \mid \xi_{k1},\xi_{k2},\cdots,\xi_{k,i-1}]$. Assume
\begin{subequations}
\begin{align}
     \mbE[\xi_{ki} \mid \xi_{k1},\xi_{k2},\cdots,\xi_{k,i-1}]&=0, \label{eq56} \\
      \sup_{k \geq 1}\sum_{i=1}^{k}\mbE[\norm{\xi_{ki}}^{2}]&<\infty,\label{eq57}\\
     \lim_{k \to \infty}S_{k}&=S, \label{eq58}\\
      \lim_{k\to \infty}\sum_{i=1}^{k}\mbE[\norm{S_{ki}-R_{ki}}]&=0,\label{eq59}\\
      \lim_{k\to \infty}\sum_{i=1}^{k}\mbE[\norm{\xi_{ki}}^{2}\mathbb{I}_{\{\norm{\xi_{ki}}>\varepsilon\}}]&=0, \quad \forall \varepsilon>0.\label{eq60}
\end{align}
\end{subequations}
  Then, $\sum_{i=1}^{k}\xi_{ki}\Rightarrow\mathrm{Normal}(0,S)$, where $\mathrm{Normal}(\mu,S)$ denotes the normal distribution with mean $\mu$ and covariance $S$.
\end{lemma}

\section{Numerical Experiments} \label{ec:Numerical Experiments}
In this section, we present empirical evaluations of our proposed algorithms. Specifically, we test our algorithms on two problems: a logistic regression with regularization and an inventory management problem. These problems are framed as stochastic composite optimization problems. We also provide detailed descriptions of our experimental setups.

\subsection{Logistics Regression with Regularization}\label{sec:Logistics Regression with Regularization}
In this subsection, we test the methods on binary classification problems where the objective function is given by a logistic loss with regularization:
\begin{equation}
	F(x) = \frac{1}{N}\sum_{i=1}^{N} \log(1+\exp(-z_{i}y_{i}^{\top}x)) + \frac{1}{2} \lambda_{1}\norm{x}_{2}^{2}+ \lambda_{2} \norm{x}_{1} . 
 \label{logistic_loss}
\end{equation}
Here $\left\{ (y_{i},z_{i}), i=1\cdots,N \right\}$ are input output data pairs, and the regularization parameters are chosen as: (i) $\lambda_{1}=\lambda_{2}=1/N$ to get a strongly convex loss;(ii) $\lambda_{1}=0$ and $\lambda_{2}=1/N$ to get a convex loss. We use the GISETTE dataset~\footnote{https://archive.ics.uci.edu/dataset/170/gisette}. The loss function \eqref{logistic_loss} can be seen as the expectation taken over a discrete uniform probability distribution defined in the dataset.

We implement four different sampling methods: (i) Continuously increase the sample size $K_{n}$ at a geometric rate and use the proximal gradient method (labeled \textbf{GEOM}), i.e.,
\begin{equation*}
	K_{n}=\lceil K_{0}(1+\gamma_{1})^{n}\rceil,
\end{equation*}
where $\gamma_{1}>0$ and $K_{0}$ is the initial sample size  (only for strongly convex objective); (ii)   Continuously increase the sample size $K_{n}$ at a polynomial rate and use the accelerated proximal gradient method (labelled \textbf{POLY}), i.e.,
\begin{equation*}
	K_{n}=\lceil K_{0}n^{\gamma_{2}}\rceil,
\end{equation*}
where $\gamma_{2}>0$ and $K_{0}$ is the initial sample size  (only for convex objective); (iii)A line search based algorithm proposed in \citeec{Franchini2023line} (labelled \textbf{LISA}); (iv) Algorithm~\ref{algorithm_adaNAPG} proposed in this paper (labelled \textbf{adaNAPG}). The parameters for these four methods are summarized in Table~\ref{tab:parameter}.

\begin{table}
\TABLE
{Parameters for Different Methods\label{tab:parameter}}
{\begin{tabular}{ccc}
\hline
& Method  & Parameters \\ \hline\up 
\multirow{3}{*}{\begin{tabular}[c]{@{}c@{}}Strongly \\ Convex\end{tabular}} & adaNAPG &  $\theta=0.9$, $\nu=5.5$          \\
                        & GEOM    &   $\gamma_{1}=0.05$, $K_{0}=2$     \\
                        & LISA    &    Same setting in \citeec{Franchini2023line}       \down \\\hline \up
\multirow{3}{*}{Convex}& adaNAPG &  $\theta=0.9$, $\nu=6.0$           \\
                         & POLY   &  $\gamma_{2}=0.01$, $K_{0}=2$          \\
                        & LISA     &   Same setting in \citeec{Franchini2023line}         \down \\ \hline
\end{tabular}

}{}
\end{table}

\subsubsection{Convergence rate}

Figure~\ref{fig:logistic_sc} and Figure~\ref{fig:logistic_cvx} illustrate the performance of different algorithms in scenarios where the loss function is strongly convex and convex, respectively. From Figure~\ref{fig:logistic_sc}, it can be observed that our proposed algorithm (\textbf{adaNAPG}) achieves the same and even better convergence effect as the \textbf{GEOM} sampling strategy while using fewer samples. Compared to \textbf{LISA}, our algorithm utilizes significantly fewer samples in the later iterations. One point worth noting in Figure~\ref{fig:logistic_sc} is that in the first 50 epochs the algorithm under \textbf{GEOM} sampling strategy exhibits significant fluctuations. This can be attributed to the small number of samples leading to larger gradient estimation errors. However, the other two adaptive sampling strategies adjust the number of samples to an appropriate quantity during the initial stages of the algorithm, ensuring its stability. In Figure~\ref{fig:logistic_cvx}, the \textbf{POLY} sampling strategy achieves a similar converge rate to \textbf{adaNAPG} but with more samples. Furthermore, since \textbf{LISA} is not an accelerated mthod, Figure~\ref{fig:logistic_cvx} also shows that the acceleration algorithm \textbf{adaNAPG} exhibits an excellent convergence rate under the convex loss scenario, which is consistent with the theoretical results.

\begin{figure}
  \FIGURE
  {
  \subcaptionbox{\label{fig:logistic_sc_sample_size}} 
  {\includegraphics[width=0.45\textwidth]{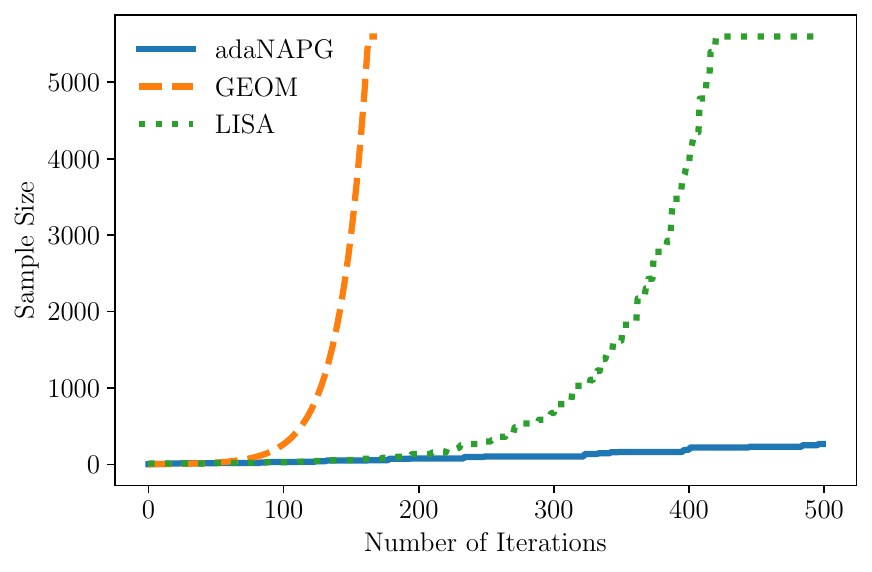}} 
  \hfill\subcaptionbox{\label{fig:logistic_sc_loss}} {\includegraphics[width=0.45\textwidth]{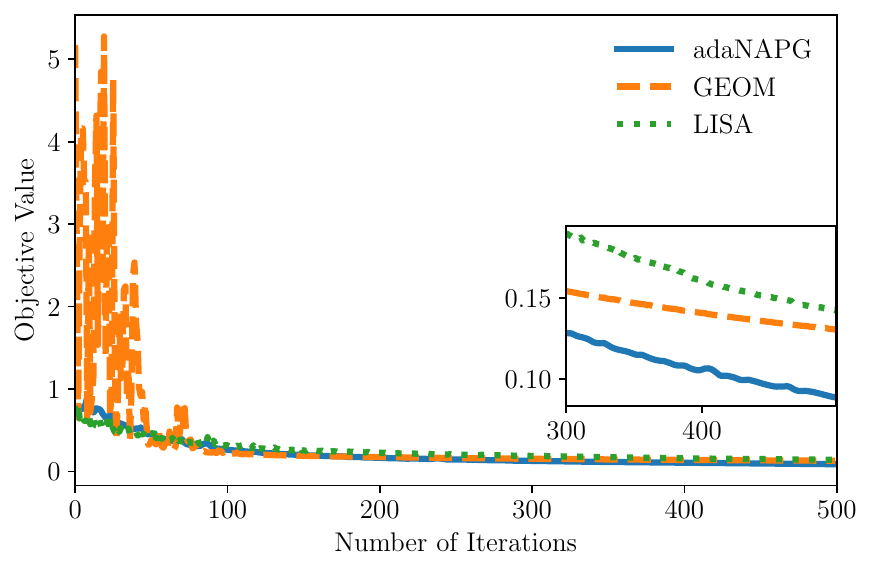}} 
  } 
  {Logistic regression with strongly convex loss. \label{fig:logistic_sc}} 
  {} 
\end{figure}

\begin{figure}
  \FIGURE
  {
\subcaptionbox{\label{fig:logistic_cvx_sample_size}} 
  {\includegraphics[width=0.45\textwidth]{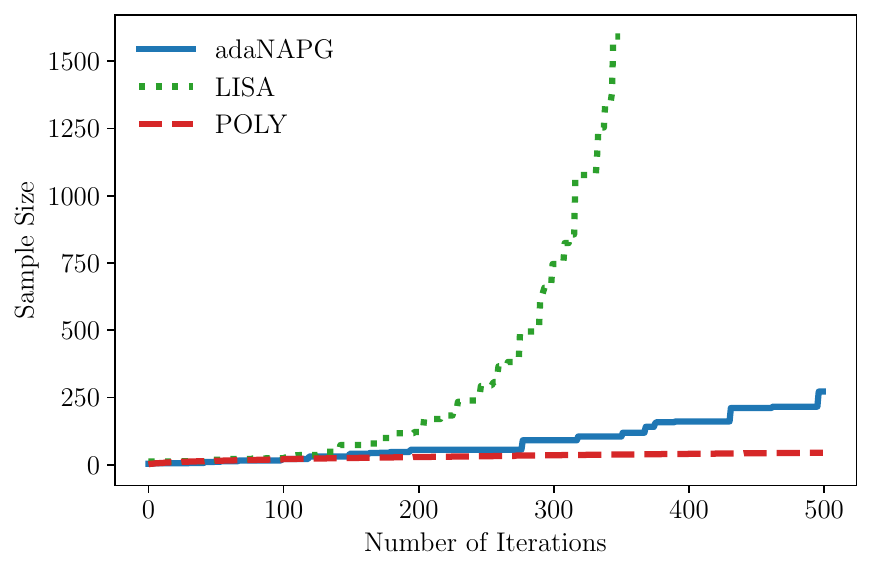}} 
  \hfill\subcaptionbox{\label{fig:logistic_cvx_loss}} {\includegraphics[width=0.45\textwidth]{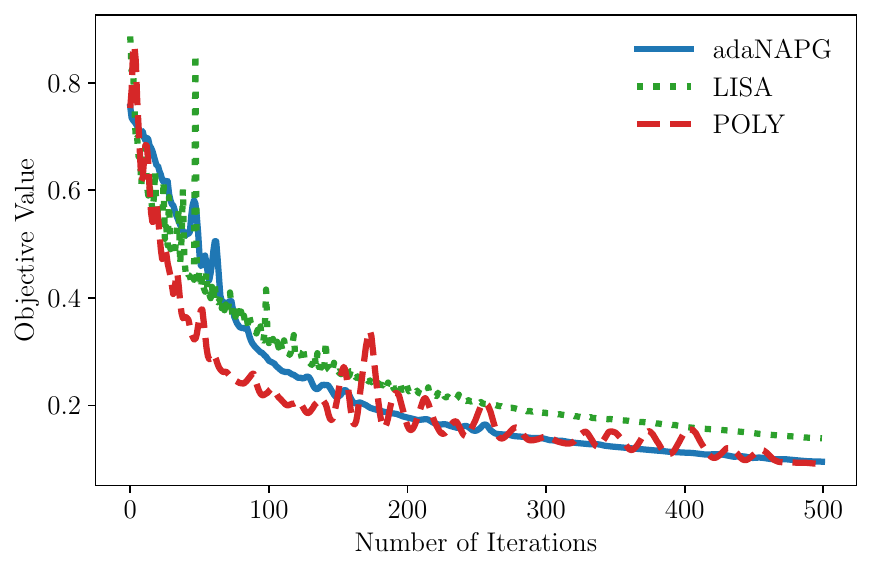}} 
  } 
  {Logistic regression with convex loss. \label{fig:logistic_cvx}} 
  {} 
\end{figure}
Figure~\ref{fig:logistic_sample_size_vs_objective} shows the efficiency of sample usage of different methods. The the x-axis measure the total samples each method has used while the y-axis is the objective value. These two figures show that Algorithm~\ref{algorithm_adaNAPG} can dynamically allocate the required number of samples for each iteration and is more efficient in sample usage, since it can achieve greater descent in the objective value given the same simulation budget.

\begin{figure}
  \FIGURE
  {
\subcaptionbox{\label{fig:logistic_sample_size_vs_objective_sc}} 
  {\includegraphics[width=0.45\textwidth]{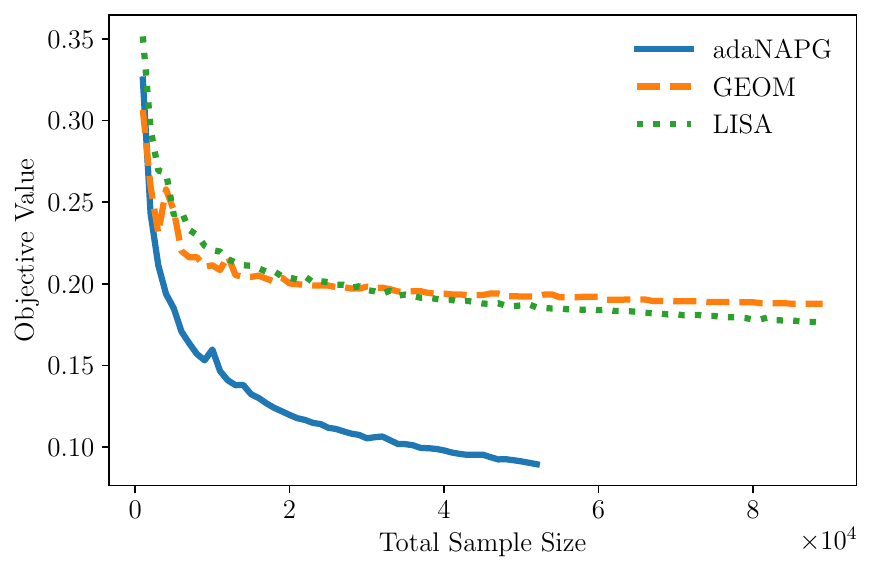}} 
  \hfill\subcaptionbox{\label{fig:logistic_sample_size_vs_objective_cvx}} {\includegraphics[width=0.45\textwidth]{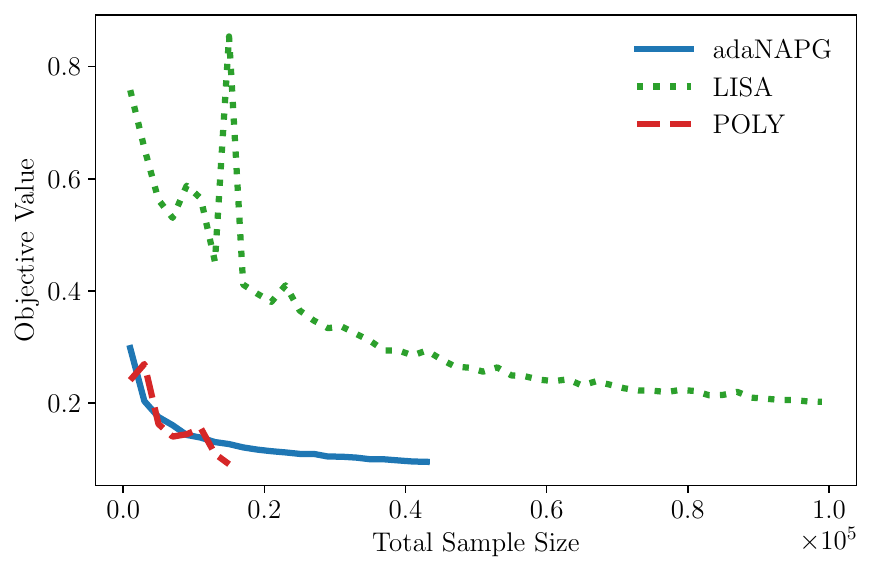}} 
  } 
  {Efficiency of sample usage. \label{fig:logistic_sample_size_vs_objective}} 
  {} 
\end{figure}

\subsubsection{Limiting distributions}
In this subsection, we show the numerical results for Theorem~\ref{CLT_adaNAPG} in Section~\ref{sec:asymptotic}. We run Algorithm~\ref{algorithm_adaNAPG} for strongly convex objectives by $1000$ independent sample paths and terminate at $T=800$. Since $x_{n}-x\opt$ is a multi-dimension vector, here we just show the first four components of $x_{n}-x\opt$. To obtain an optimal solution $x\opt$, we solve the logistic regression problem~\eqref{logistic_loss} by running FISAT \citepec{Beck2017First} with full gradient. 
Figure~\ref{fig:hist_adaNAPG} show the histogram of the sequence $x_{n}$ generated by Algorithm~\ref{algorithm_adaNAPG}. The fitted normal distributions are also showed (the red curve). It can be seen that the error is normally distributed, which is consist with the theory in Section~\ref{sec:asymptotic}.

We also compute the root-mean-square error (RMSE) of the sample sequence. Noting that for each algorithm we run $1000$ sample paths and terminate at $n=800$, the RMSE for each step is computed as
\begin{equation*}
    \textrm{RMSE}_{n} = \sqrt{\frac{1}{1000}\sum_{j=1}^{1000}\norm{x_{n}^{(j)}-x\opt}^{2}},
\end{equation*}
where $x_{800}^{(j)}$ is the $j$-th sample path of $x_{800}$. Figure~\ref{fig:logistic_RMSE_adaNAPG} shows the result for Algorithm~\ref{algorithm_adaNAPG}, where the y-axis measure the $\log(RMSE)$. We also compute the linear relationship between $\log(RMSE)$ and $T$ (the solid line).

\begin{figure}
  \FIGURE
  {
    \subcaptionbox{Histogram of some components of $x_{800}$\label{fig:hist_adaNAPG}} 
  {\includegraphics[width=0.45\textwidth]{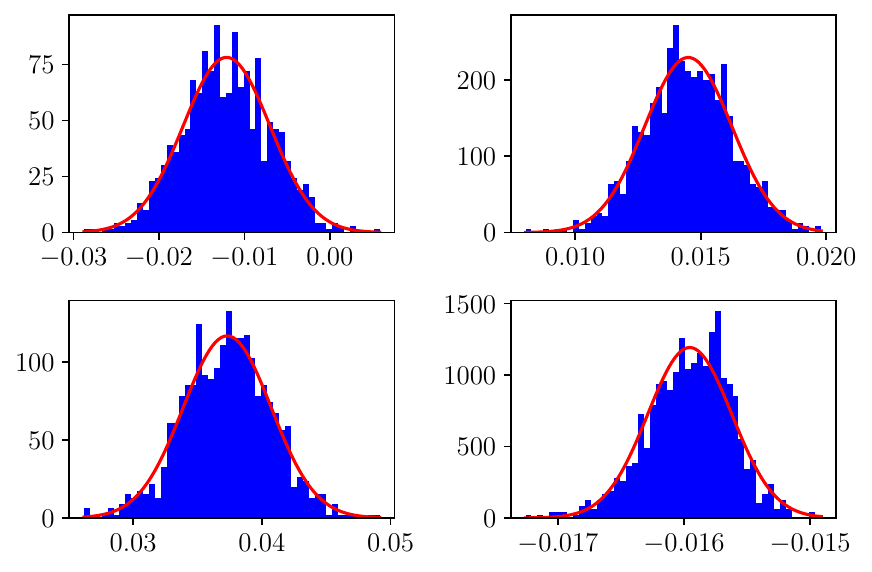}} 
  \hfill\subcaptionbox{RMSE \label{fig:logistic_RMSE_adaNAPG}} {\includegraphics[width=0.45\textwidth]{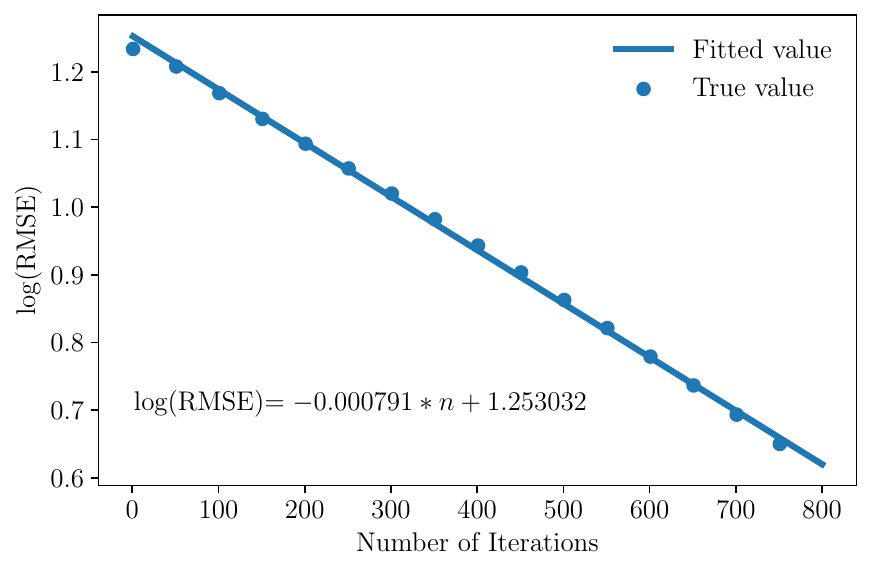}}   
  } 
  {Histogram and RMSE for Algorithm~\ref{algorithm_adaNAPG}.\label{fig:CLT_adaNAPG}} 
  {} 
\end{figure}

\subsection{Inventory Control Problem}

In this subsection, we consider a large-scale inventory problem with a general network structure similar to \citeec{Wang2023Large}. The optimization problem is given by
\begin{equation}\label{inventory}
    \min_{S} \; \mbE\left[\sum_{t=1}^{T}C_{t}(S)\right] + \lambda \norm{S}_{1}.
\end{equation}
Here $S$ is the base-stock level,  $C_{t}(S)$ is the total cost for period $t$ with base-stock level $S$ and $\lambda>0$ is a regularization parameter. Noticing that the convexity the total cost functions of inventory problems depends on the supply chain structure and the model setting, 
Problem~\eqref{inventory} may not always be convex. 

In our experiment the total time duration is set to 50, the number of network nodes is set to 1000, and the random demand follows a normal distribution. We use the BP Algorithm proposed in \citeec{Wang2023Large} to estimate the gradient of the first term in Problem~\eqref{inventory}. Two sampling strategies are adopted: (i) a fixed number of samples is taken at each step (labeled \textbf{FIX}) with $K=50,100$; (ii) Algorithm~\ref{algorithm_adaNAPG} (labeled \textbf{adaNAPG}). 

Figure~\ref{fig:inventory_1000} presents the convergence results of various policies. Figure~\ref{fig:inventory_1000_objective} demonstrates that our adaptive method converges significantly faster than the fixed sample size strategy and achieves a lower objective value. Figure~\ref{fig:inventory_1000_total_objective} evaluates the efficiency of sample usage, further illustrating the benefits of Algorithm~\ref{algorithm_adaNAPG} in the allocation of simulation resources.
\begin{figure}
  \FIGURE
  {
\subcaptionbox{\label{fig:inventory_1000_objective}} 
  {\includegraphics[width=0.45\textwidth]{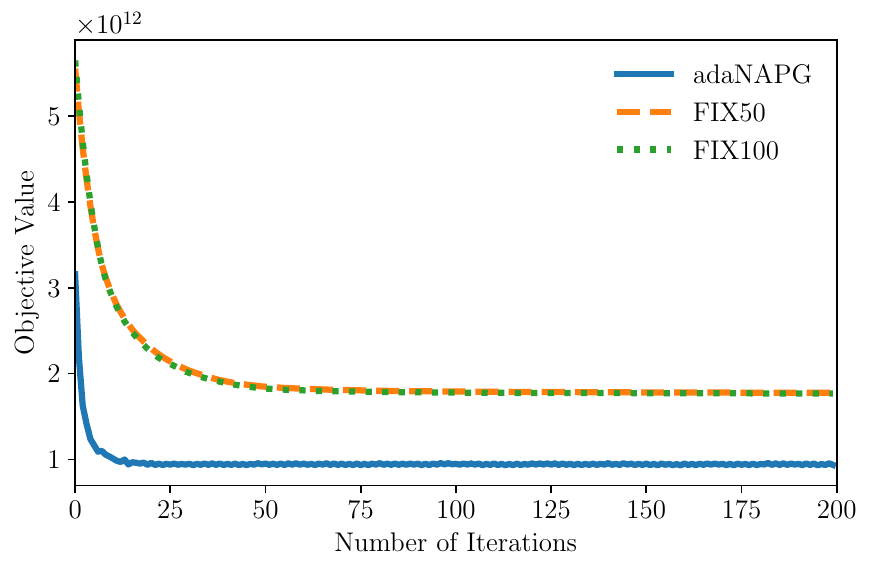}} 
  \hfill\subcaptionbox{\label{fig:inventory_1000_total_objective}} {\includegraphics[width=0.45\textwidth]{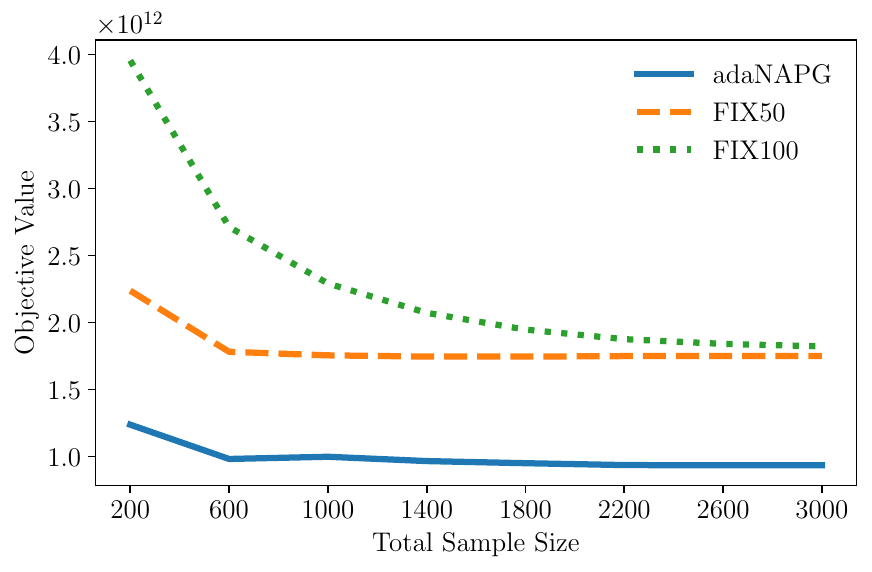}} 
  } 
  {Inventory control problem with 1000 nodes. \label{fig:inventory_1000}} 
  {} 
\end{figure}

\subsection{Discussion on Assumption~\ref{assum_clt_noise}}\label{appendix:remark_clt_noise}

In this section, we will further examine the reasonableness of condition~\eqref{eq51} in Assumption~\ref{assum_clt_noise} through two numerical examples. We begin by revisiting condition~\eqref{eq51} here, 
\begin{equation}\label{eq51r}
            \lim_{n \to \infty}  \rho^{-n} \mbE[\tilde{w}_{n}\tilde{w}_{n}^{\top}]  = \tilde{S}_{0},
        \end{equation}
 where $\rho$ is defined in Theorem~\ref{adaNAPG_convergence} as the convergence rate parameter for the strongly convex objective, and $\tilde{S}_{0}$ is a semi-definite matrix. Setting $ W_{n}= \rho^{-n} \mbE[\tilde{w}_{n}\tilde{w}_{n}^{\top}]$, Condition~\eqref{eq51r} necessitates that the sequence $\{W_{n}\}$ converges to some matrix (if the limit exists, it will naturally be semi-definite). Since the limiting matrix $\tilde{S}_{0}$ cannot be known analytically in advance, we define the following gap function to evaluate the sequence's convergence, 
\begin{equation*}
    \Delta W_{n} = \norm{W_{n+1}-W_{n}}_{F},
\end{equation*}
which measures the difference between two successive steps. When $\Delta W_{n} \to 0$, we consider that the sequence converges to some extent (but this may not be a sufficient condition for the sequence convergence).

\subsubsection{Parameter Estimation Problem}\label{sec:app21}

Here we consider a parameter estimation problem, where aiming to estimate the unknown $d$-dimensional parameter $x\opt$ based on the gathered scalar measurements given by $l= u^{\top}x\opt+v$, where $u\in \mbR^{d}$ denotes the regression vector and $v\in \mbR$ denotes the local observation noise. Assume that $u$ and $v$ are mutually independent i.i.d. Gaussian sequences with distributions $\Normal(0,R_{u})$ and $\Normal(0,\sigma_{v}^{2})$, respectively. Suppose the covariance matrix $R_u$ is positive definite. Then, we might model the parameter estimation problem as the following stochastic composite optimization problem:
\begin{equation*}
    \min_{x \in \mbR^{d}} \; \mbE[\norm{l-u^{\top}x}^{2}] + \lambda \norm{x}_{1},
\end{equation*}
where $\lambda >0$. Let $f(x)=\mbE[\norm{l-u^{\top}x}^{2}]$ and $h(x)=\lambda \norm{x}_{1}$. By basic calculus, we have $f(x)=(x-x\opt)^{\top}R_{u}(x-x\opt) + \sigma_{v}^{2}$ and $\nabla f(x)=2R_{u}(x-x\opt)$. Suppose that we can observe $u$ and $l$ then the noisy observation of the gradient $\nabla f(x)$ might be constructed as $g(x,u,v) = 2(uu^{\top}x-lu)$. In this scenario, we initialize $d=10$, with $\lambda=3$ and $R_{u}$ being randomly generated. For Algorithm~\ref{algorithm_adaNAPG}, we configure parameters $\theta=1.2$ and $\nu=1.4$. Quantities $L$ and $\mu$ are computed using $\lambda_{max}(R_{u})$ and $\lambda_{min}(R_{u})$, respectively. The algorithm runs for $200$ iterations, and during each iteration, we utilize $20$ sample paths to estimate $W_n$ before calculating $\Delta W_n$. Figure~\ref{fig:sr_deltaw} illustrates that the gap $\Delta W_{n}$ between the successive steps decreases from approximately $7000$ at the start to around $3$ after $200$ iterations. Given that the matrix is $10\times 10$ and the Frobenius norm is applied, this gap is, indeed, very small.

\subsubsection{Logistic Regression}
The regression problem follows the same setting as Section~\ref{sec:Logistics Regression with Regularization}. Figure~\ref{fig:l_deltaw} shows the gap $\Delta W_{n}$ between the successive steps of the sequence $\{W_{n}\}$ for the logistic regression problem  also decreases to $0$ after $500$ iterations. This result is consistent with the parameter estimation problem, demonstrating that Condition~\ref{eq51r} is reasonable.

\begin{figure}
  \FIGURE
  {
\subcaptionbox{$\Delta W_{n}$ for the parameter estimation problem.\label{fig:sr_deltaw}}{\includegraphics[width=0.45\textwidth]{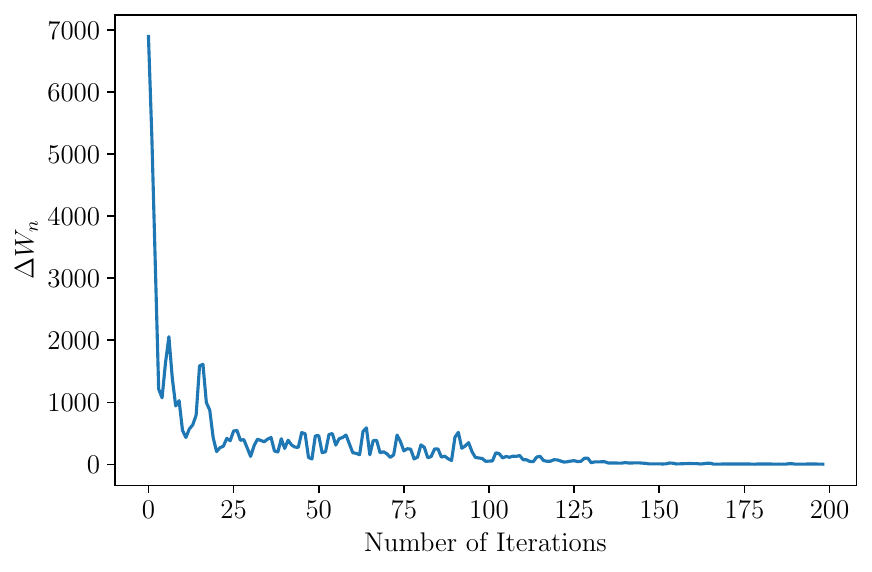}} 
  \hfill\subcaptionbox{$\Delta W_{n}$ for the logistic regression problem.\label{fig:l_deltaw}} {\includegraphics[width=0.45\textwidth]{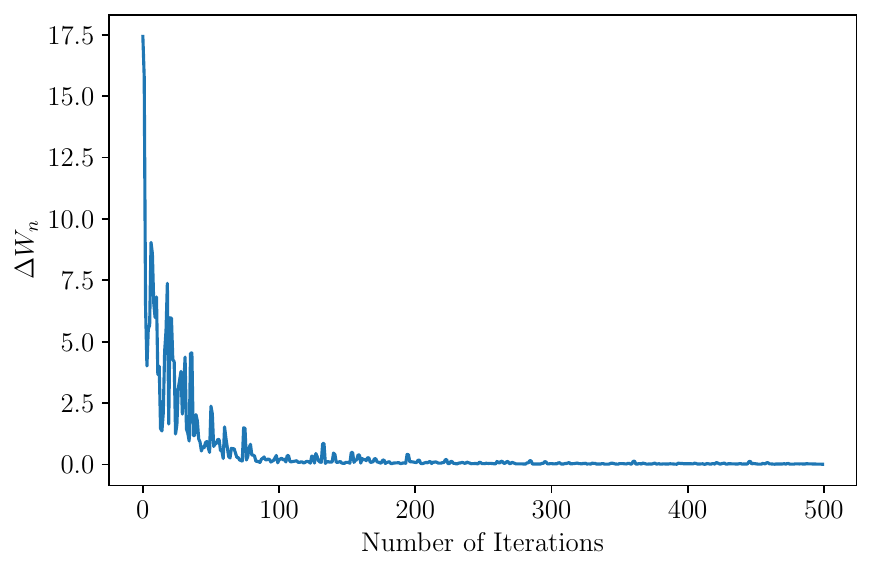}} 
  } 
  {Numerical evidences for Assumption~\ref{assum_clt_noise}} 
  {} 
\end{figure}

\bibliographystyleec{informs2014}
\bibliographyec{sample}

\end{document}